\documentclass[smallcondensed,numbook,envcountsame]{svjour3}
\usepackage{color}
\usepackage{amsmath}
\usepackage{amsfonts}
\usepackage{amssymb}
\usepackage{mathrsfs}
\usepackage{tabularx}

\usepackage{tikz, ifthen}
\usetikzlibrary{shapes.misc}

\newcommand{\undertilde}[1]{\ensuremath{\mathord{\vtop{\ialign{##\crcr
   $\hfil\displaystyle{#1}\hfil$\crcr\tikzstyle{largegraph}=[
  every node/.style={circle,draw,fill=black,inner sep=0pt,minimum width=6pt},
  every path/.append style={semithick}
]\noalign{\kern1.5pt\nointerlineskip}
   $\hfil\tilde{}\hfil$\crcr\noalign{\kern1.5pt}}}}}}

\tikzstyle{smallgraph}=[
  every node/.style={circle,draw,fill=black,inner sep=0pt,minimum width=0pt},
  every path/.append style={semithick}
]

\tikzstyle{largegraph}=[
  every node/.style={circle,draw,fill=black,inner sep=0pt,minimum width=2pt},
  every path/.append style={semithick}
]
\pgfmathsetmacro{\threevradius}{0.375*sqrt(2/3)}

\pgfmathsetmacro{\threevradius}{0.375*sqrt(2/3)}

\newcommand\flower[1]{%
\foreach \m in {1,..., #1} {\node (a\m) at (\m,1) {};};
\foreach \m in {1,..., #1} {\node (b\m) at (\m,0.5) {};};
\foreach \m in {1,..., #1} {\node (c\m) at (\m,1.5) {};};
\foreach \m in {1,..., #1} {\node (d\m) at (\m,2) {};};
\foreach \n in {1,...,#1} {\draw (a\n) -- (b\n);}
\foreach \n in {1,...,#1} {\draw (a\n) -- (c\n);}
\draw (b1) -- (b#1);
\draw (c1) -- (c#1);
\draw (d1) -- (d#1);
\foreach \n in {1,...,#1} {\draw[bend right=90] (a\n) to (d\n);}
\draw[dotted] (b1) to (0.5,0.5);
\draw[dotted] (c1) to (0.5,1.5);
\draw[dotted] (d1) to (0.5,2);
\draw[dotted] (b#1) to (#1+0.5,0.5);
\draw[dotted] (c#1) to (#1+0.5,1.5);
\draw[dotted] (d#1) to (#1+0.5,2);
\node[white] [label=right:{$a$}] at (0,1) {};
\node[white] [label=right:{$b$}] at (0,0.5) {};
\node[white] [label=right:{$c$}] at (0,1.5) {};
\node[white] [label=right:{$d$}] at (0,2) {};
}
\newcommand\flowerend[1]{%
\foreach \m in {1,..., #1} {\node (a\m) at (\m,1) {};};
\foreach \m in {1,..., #1} {\node (b\m) at (\m,0.5) {};};
\foreach \m in {1,..., #1} {\node (c\m) at (\m,1.5) {};};
\foreach \m in {1,..., #1} {\node (d\m) at (\m,2) {};};
\foreach \n in {1,...,#1} {\draw (a\n) -- (b\n);}
\foreach \n in {1,...,#1} {\draw (a\n) -- (c\n);}
\draw (b1) -- (b#1);
\draw (c1) -- (c#1);
\draw (d1) -- (d#1);
\foreach \n in {1,...,#1} {\draw[bend right=90] (a\n) to (d\n);}
\draw[dotted] (b1) to (0.5,0.5);
\draw[dotted] (c1) to (0.5,1.5);
\draw[dotted] (d1) to (0.5,2);
\draw[dotted] (b#1) to (#1+0.5,0.5);
\draw[dotted] (c#1) to (#1+0.5,2);
\draw[dotted] (d#1) to (#1+0.5,1.5);
\node[white] [label=right:{$a$}] at (0,1) {};
\node[white] [label=right:{$b$}] at (0,0.5) {};
\node[white] [label=right:{$c$}] at (0,1.5) {};
\node[white] [label=right:{$d$}] at (0,2) {};
}

\newcommand\inS[2]{
\tikzset{inS/.style={circle,draw,fill=black,inner sep=0pt,minimum width=8pt}}
\ifthenelse{\equal{#1}{a}}{\node[inS] at (#2,1) {};}{};
\ifthenelse{\equal{#1}{b}}{\node[inS] at (#2,0.5) {};}{};
\ifthenelse{\equal{#1}{c}}{\node[inS] at (#2,1.5) {};}{};
\ifthenelse{\equal{#1}{d}}{\node[inS] at (#2,2) {};}{};
}

\newcommand\inSS[2]{
\tikzset{inS/.style={rectangle,draw,fill=black,inner sep=0pt,minimum width=8pt,minimum height=8pt}}
\ifthenelse{\equal{#1}{a}}{\node[inS] at (#2,1) {};}{};
\ifthenelse{\equal{#1}{b}}{\node[inS] at (#2,0.5) {};}{};
\ifthenelse{\equal{#1}{c}}{\node[inS] at (#2,1.5) {};}{};
\ifthenelse{\equal{#1}{d}}{\node[inS] at (#2,2) {};}{};
}

\newcommand\inSa[2]{
\tikzset{inS/.style={circle,draw,fill=white,inner sep=0pt,minimum width=8pt}}
\ifthenelse{\equal{#1}{a}}{\node[inS] at (#2,1) {};}{};
\ifthenelse{\equal{#1}{b}}{\node[inS] at (#2,0.5) {};}{};
\ifthenelse{\equal{#1}{c}}{\node[inS] at (#2,1.5) {};}{};
\ifthenelse{\equal{#1}{d}}{\node[inS] at (#2,2) {};}{};
}

\newcommand\notS[2]{
\tikzset{inS/.style={cross out,draw,minimum size=8*(\pgflinewidth), inner sep=0pt, outer sep=0pt}}
\ifthenelse{\equal{#1}{a}}{\node[inS] at (#2,1) {};}{};
\ifthenelse{\equal{#1}{b}}{\node[inS] at (#2,0.5) {};}{};
\ifthenelse{\equal{#1}{c}}{\node[inS] at (#2,1.5) {};}{};
\ifthenelse{\equal{#1}{d}}{\node[inS] at (#2,2) {};}{};
}

\newcommand\notSa[2]{
\tikzset{inS/.style={strike out,draw,minimum size=8*(\pgflinewidth), inner sep=0pt, outer sep=0pt}}
\ifthenelse{\equal{#1}{a}}{\node[inS] at (#2,1) {};}{};
\ifthenelse{\equal{#1}{b}}{\node[inS] at (#2,0.5) {};}{};
\ifthenelse{\equal{#1}{c}}{\node[inS] at (#2,1.5) {};}{};
\ifthenelse{\equal{#1}{d}}{\node[inS] at (#2,2) {};}{};
}

\newcommand\patt[2]{
\tikzset{inS/.style={circle,draw,fill=white,inner sep=0pt,minimum width=0pt}}
\node[white] [label=above:{#2}] at (#1,-0.1) {};
}

\newcommand\lab[2]{
\node[white] [label=above:{#2}] at (#1,2.2) {};
}

\newcommand\genlab[1]{
\foreach \n in {1,...,#1} {\node[white] [label=above:{$J^\n$}] at (\n,2.2) {};};
}

\renewenvironment{proof}{\par {\sc {\bf Proof.}\hskip 5pt}}{\hfill \qed \par}

\begin{document}

\title{Variants of the Domination Number for Flower Snarks}
\author{Ryan Burdett, Michael Haythorpe, Alex Newcombe}

\institute{Ryan Burdett
\at Flinders University, 1284 South Road, Tonsley Park, SA, Australia\\
\email{ryan.burdett@flinders.edu.au} \and Michael Haythorpe (Corresponding author)
\at Flinders University, 1284 South Road, Tonsley Park, SA, Australia, Ph: +61 8 8201 2375, Fax: +61 8 8201 2904\\
\email{michael.haythorpe@flinders.edu.au} \and Alex Newcombe
\at Flinders University, 1284 South Road, Tonsley Park, SA, Australia\\
\email{alex.newcombe@flinders.edu.au}} \maketitle

\begin{abstract}We consider the flower snarks, a widely studied infinite family of 3--regular graphs. For the Flower snark $J_n$ on $4n$ vertices, it is trivial to show that the domination number of $J_n$ is equal to $n$. However, results are more difficult to determine for variants of domination. The Roman domination, weakly convex domination, and convex domination numbers have been determined for flower snarks in previous works. We add to this literature by determining the independent domination, 2-domination, total domination, connected domination, upper domination, secure domination and weak Roman domination numbers for flower snarks.

\end{abstract}


\section{Introduction}

Consider a graph $G$ containing the vertex set $V$ and edge set $E$. A subset of the vertices $S \subset V$ is said to be a {\em dominating set} if every vertex in $V \setminus S$ is adjacent to at least one vertex in $S$. Then, the {\em domination problem} is to determine the size of the smallest dominating set in a given graph $G$, which is known as the {\em domination number} of $G$ and is denoted by $\gamma(G)$.

There are obvious real-world applications for the domination problem. For example, suppose that the vertices of a graph correspond to locations in a secure site, and that each location needs to remain under observation by guards. If a guard at one location is able to simultaneously observe another, there is an edge between the corresponding vertices. Then, by placing guards at the sites corresponding to any dominating set, all locations are under observation. Clearly, it is desirable to do so with as few guards as possible.

However, in many real-world applications, the definition of domination may be unsuitable in some way, and so variants of domination have been described to handle such cases. Continuing the example in the previous paragraph, suppose that the guards carry out their observation from the top of large towers. The vantage point from these towers enables them to observe adjacent locations, but does not permit them to observe their own location. Then, their location would need to be observed by a guard at an adjacent location. In such a situation, we are seeking not a dominating set, but a {\em total dominating set}, which we define formally now along with several other variants of dominating sets that we will consider in this manuscript.

\begin{definition}Consider a dominating set $S \subseteq V$. Then:

\begin{itemize}\item $S$ is a {\em (weakly) convex dominating set} if $S$ is (weakly) convex in $G$.
\item $S$ is an {\em independent dominating set} if $S$ is an independent set in $G$.
\item $S$ is a {\em minimal dominating set} if any proper subset of $S$ is not a dominating set.
\item $S$ is a {\em 2-dominating set} if any vertex $v \not\in S$ is adjacent to at least two vertices in $S$.
\item $S$ is a {\em total dominating set} if every vertex in $V$ is adjacent to at least one vertex in $S$.
\item $S$ is a {\em connected dominating set} if the subgraph of G induced by the vertices in $S$ is connected.
\item $S$ is a {\em secure dominating set} if, for every vertex $v \in V \setminus S$, there exists a vertex $w \in S$ such that $vw \in E$, and $(S \setminus \{w\}) \cup \{v\}$ is a dominating set.
\end{itemize}\end{definition}

Analogously to the domination number, we define $\gamma_{wcon}(G)$ to be the weakly convex domination number, $\gamma_{con}(G)$ to be the convex domination number, $i(G)$ to be the independent domination number, $\gamma_2(G)$ to be the 2-domination number, $\gamma_t(G)$ to be the total domination number, $\gamma_c(G)$ to be the connected domination number, and $\gamma_s(G)$ to be the secure domination number. We also define the upper domination number, $\Gamma(G)$, as follows.

\begin{definition}The {\em upper domination number}, denoted by $\Gamma(G)$, is equal to the size of the largest minimal dominating set.\end{definition}

In addition, we consider two more variants of domination. Suppose that for our graph $G$, we have a function $f : V \rightarrow \{0, 1, 2\}$. Then, the {\em weight} of $f$, denoted $w(f)$, is equal to $\sum_{v \in V} f(v)$.

\begin{definition}$f$ is a {\em Roman dominating function} on $G$ if it satisfies the condition that every vertex $v$ for which $f(v) = 0$ is adjacent to at least one vertex $w$ for which $f(w) = 2$.\end{definition}

In the following definition, we say that a vertex $v$ is {\em undefended} with respect to $f$ if $f(v) = 0$ and for every vertex $w$ adjacent to $v$, $f(w) = 0$.

%\begin{definition}$f$ is a {\em weak Roman dominating function} on $G$ if it satisfies the condition that any vertex $v$ for which $f(v) = 0$ is adjacent to at least one vertex $w$ satisfying the following: if we define a new function $f' : V \rightarrow \{0, 1, 2\}$ defined by $f'(v) = 1$, $f'(w) = f(w)-1$, and $f'(u) = f(u)$ for all $u \neq \{v,w\}$, then there are no undefended vertices with respect to $f'$.\end{definition}

\begin{definition}$f$ is a {\em weak Roman dominating function} on $G$ if, for every vertex $v$ such that $f(v) = 0$, there is at least one vertex $w$, adjacent to $v$ and satisfying the following: if we define a new function $g : V \rightarrow \{0, 1, 2\}$ defined by $g(v) = 1$, $g(w) = f(w)-1$, and $g(u) = f(u)$ for all $u \neq \{v,w\}$, then there are no undefended vertices with respect to $g$.\label{def-wrd}\end{definition}

Then, the Roman domination number $\gamma_R(G)$ is the minimum weight over all Roman dominating functions, and equivalently for the weak Roman domination number $\gamma_r(G)$. It is worth noting that if we further demand that $f(v) \leq 1$ for all $v \in V$ then weak Roman domination is equivalent to secure domination. Hence, it is clear that $\gamma_r(G) \leq \gamma_s(G)$.

Domination numbers are known for some infinite families of graphs. Other than trivial results such as for paths or cycles, perhaps the most famous result is the sprawling effort over a 27 year period \cite{jacobson,hare,chang,spalding,alanko,goncalves} to provide a complete characterisation of domination numbers for grid graphs $G(n,m)$ of all possible sizes, consisting of 23 special cases before settling into a standard formula for $n,m \geq 16$. Other results for domination include generalized Petersen graphs \cite{yan,liu,xueliang}, Cartesian products involving cycles \cite{pavlic,crevals,agustin}, King graphs \cite{yaglom}, Latin square graphs \cite{pouyandeh}, hypercubes \cite{arumugam}, Sierpi\'{n}ski graphs \cite{shan}, Kn\"{o}del graphs \cite{xueliang2}, and various graphs from chemistry \cite{mojdeh,quadras}, among others.

However, far fewer results are known for variants of domination beyond trivial results such as for paths, cycles, stars, wheels, or complete graphs. We summarise the most noteworthy of these results for the variants of domination considered in this paper. For upper domination, results are known for various graphs based on chessboards \cite{burcroff,fricke,hedetniemi,wallis,yaglom}. For total domination, results are known for some grid graphs $G(n,m)$ (for $n \leq 6$) \cite{gravier,klobucar}, Kn\"{o}del graphs \cite{mojdeh2}, and various graphs from chemistry \cite{mojdeh}. For connected domination, results are known for trees \cite{sampathkumar}, circulant graphs \cite{shobana}, Centipede graphs \cite{vijayan}, and various graphs from chemistry \cite{mojdeh}. For weak Roman domination, results are known for various graphs based on chessboards \cite{pushpam2}, Cartesian products involving complete graphs \cite{valveny}, and Helm graphs and Web graphs \cite{lai}. For secure domination, results are known for some grid and torus graphs (for $n \leq 3$) \cite{securedom}, various Cartesian products of stars, cycles, paths and complete graphs \cite{valveny}, and middle graphs \cite{pushpam}.

In recent works, the following results were shown for flower snarks (which will be defined in the next section).

\begin{theorem}[Maksimovic et al. (2018) \cite{maksimovic}, Kratica et al. (2020)\cite{kratica}]Consider the flower snark $J_n$, for $n \geq 3$. Then we have $\gamma_R(J_n) = \gamma_{wcon}(J_n) = 2n$, and $\gamma_{con}(J_n) = 4n$.\end{theorem}

We now add to the above literature by proving the following additional results for flower snarks:

\begin{theorem}Consider the flower snark $J_n$, for $n \geq 3$. Then,

$$\gamma(J_n) = i(J_n) = n, \hspace*{0.75cm} \gamma_s(J_n) = \gamma_r(J_n) = \left\lceil\frac{3n+1}{2}\right\rceil, \hspace*{0.75cm} \gamma_2(J_n) = \left\{\begin{aligned}\left\lceil\frac{5n}{3}\right\rceil\;, & \mbox{ if } n \neq 1\mbox{ mod }3,\\\frac{5n+4}{3}, & \mbox{ if } n = 1\mbox{ mod }3,\end{aligned}\right.$$
$$\hspace*{0.75cm}\gamma_t(J_n) = \left\{\begin{aligned}\left\lceil\frac{3n}{2}\right\rceil\;, & \mbox{ if } n \neq 2\mbox{ mod }4,\\\frac{3n}{2} + 1, & \mbox{ if } n = 2\mbox{ mod }4,\end{aligned}\right. \hspace*{2.4cm} \gamma_c(J_n) = \Gamma(J_n) = \left\{\begin{aligned}2n\;\;\;, & \mbox{ if } n\mbox{ is even,}\\2n-1, & \mbox{ if } n\mbox{ is odd.}\end{aligned}\right.\hspace*{1.0cm}$$
\label{thm-overall}\end{theorem}

In most cases, the proofs will be by induction, and hence it will be necessary to first prove the results for some number of base cases. Rather than provide these proofs here, we will simply use mixed-integer linear programming formulations of each variant of domination from literature to handle the base cases.

\section{Flower Snarks}

The {\em chromatic index} of a graph is the minimum required number of colours to color the edges of the graph, such that no two incident edges have the same colour. By Vizing's theorem, it is known that all 3--regular graphs have chromatic index 3 or 4. The latter case is rare, and simple, connected, bridgeless 3--regular graphs with chromatic index 4 are called {\em snarks}. It is common to further add the restriction that the girth should be at least 5, with such graphs known as {\em nontrivial snarks}. Flower snarks \cite{isaacs}, discovered by Isaacs in 1975, were the first known infinite family of nontrivial snarks, and are denoted by $J_n$. They are defined as follows.

\begin{definition}[Flower snarks]For $n \geq 3$, take the union of $n$ copies of $K_{1,3}$. Denote the degree 3 vertex in the $i$-th copy as $a^i$, and the other three vertices in the $i$-th copy as $b^i$, $c^i$ and $d^i$. Then construct an $n$-cycle through vertices $b^1, b^2, \hdots, b^n$, and a $2n$-cycle through vertices $c^1, c^2, \hdots, c^n, d^1, d^2, \hdots, d^n$.\end{definition}

In order to have the properties of a nontrivial snark, $n$ must be odd and $n \geq 5$. However, for other values of $n \geq 3$ a 3--regular graph is nonetheless obtained by this construction. In this paper we will consider all $n \geq 3$, and for convenience we will refer to all of them as flower snarks.

For the remainder of this document, we will consider various kinds of dominating sets of flower snarks. As such, it is convenient to go over some brief terms and notation here. Recall that $J_n$ contains $n$ copies of $K_{1,3}$. We will refer to the $i$-th copy as $J^i$, and its four vertices as $a^i$, $b^i$, $c^i$ and $d^i$. Note that this notation does not include $n$, as typically $n$ will be fixed in our considerations. Also, we will say that a copy $J^i$ has {\em weight} $k$ in a dominating set $S$, if $S$ contains $k$ vertices from $J^i$. We will use the term {\em pattern} to describe a sequence of weights of consecutive copies of $J_n$ in $S$. For example, if we say that $S$ contains the pattern 121, it means there is a set of three consecutive copies $J^{i-1}, J^i, J^{i+1}$ which have weights 1, 2 and 1 respectively in $S$. Further to this, in a set of consecutive copies of $J_n$, we will use the term {\em configuration} to describe the specific allocation of its vertices to $S$. Finally, we will define $w^S_i$ to be equal to the number of copies with weight $i$ in $S$.

Flower snarks can be visualised in various ways. A common method is to distribute the copies of $K_{1,3}$ in a circle, using curved edges for the $2n$-cycle and straight edges elsewhere. We display one such drawing of $J_n$ in part (a) of Figure \ref{fig-flower}, for $n = 7$. However, for our purposes it will be convenient to focus only on a small section of a flower snark at a time, and so we will use the drawing style displayed in part (b) of Figure \ref{fig-flower}, with each copy displayed vertically. As indicated in Figure \ref{fig-flower} we will assume that in these drawings, the bottom vertex in copy $J^i$ is $b^i$, followed by $a^i$, $c^i$ and $d^i$. When useful to avoid confusion, a label will be given above each copy.

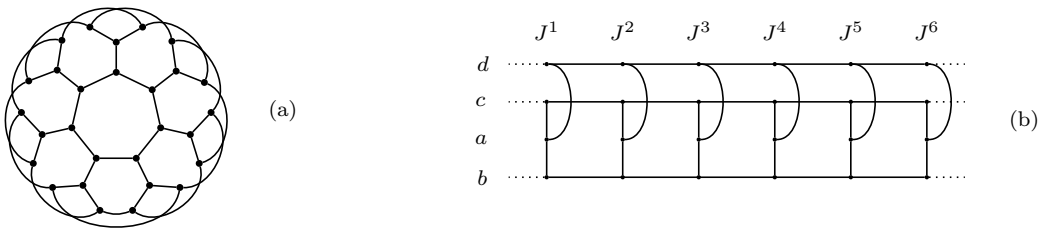
\begin{figure}[h!]\begin{center}\begin{tikzpicture}[largegraph,rotate=206,scale=0.5]
  \pgfmathtruncatemacro{\ang}{90}
  \begin{scope}[rotate=\ang]
    \node (0) at (2,0) {};
    \path (0) -- ++(180:.8) node (b0) {};
    \path (0) -- ++(60:.8) node (c0) {};
    \path (0) -- ++(300:.8) node (d0) {};
    \draw (0) -- (b0);
    \draw (0) -- (c0);
    \draw (0) -- (d0);
  \end{scope}

\foreach \n in {1,...,6}{
  \pgfmathtruncatemacro{\ang}{90+\n*360/7}
  \begin{scope}[rotate=\ang]
    \node (\n) at (2,0) {};
    \path (\n) -- ++(180:.8) node (b\n) {};
    \path (\n) -- ++(60:.8) node (c\n) {};
    \path (\n) -- ++(300:.8) node (d\n) {};
    \draw (\n) -- (b\n);
    \draw (\n) -- (c\n);
    \draw (\n) -- (d\n);
  \end{scope}
}
\foreach \n in {0,...,6}{
  \pgfmathtruncatemacro{\nextn}{mod(\n+1,7)}
  \draw (b\n) -- (b\nextn);
}
\foreach \n in {0,...,5}{
  \pgfmathtruncatemacro{\nextn}{mod(\n+1,7)}
  \draw[bend right=70] (c\n) to (c\nextn);
}
\draw[bend right=20] (c6) to (d0);
\foreach \n in {0,...,5}{
  \pgfmathtruncatemacro{\nextn}{mod(\n+1,7)}
  \draw[bend right=70] (d\n) to (d\nextn);
}
\draw[bend right=80] (d6) to (c0);

\node[white] [label=right:{(a)}] at (-3.5,1.5) {};
\end{tikzpicture} \hspace*{2cm} \begin{tikzpicture}[smallgraph]\flower{6} \genlab{6} \node[white] [label=right:{}] at (0.5,-0.3) {};\node[white] [label=right:{(b)}] at (7,1.25) {};\end{tikzpicture}
\caption{In part (a) a common drawing of the flower snark $J_7$. In part (b), a section of a larger flower snark consisting of six copies. \label{fig-flower}} \end{center}\end{figure}

It is worth noting that, when viewing only a section of $J_n$, vertices $b^i$, $c^i$ and $d^i$ are all essentially equivalent, and this will be useful in simplifying many of the upcoming proofs. In a global sense this is not the case, as there is a ``twist" in the final copy in which $c^n$ links to $d^1$, and $d^n$ links to $c^1$. However, due to the symmetry of the flower snark, when viewing any copy locally we can choose to relabel the vertices so that this twist occurs elsewhere in the graph. Hence, in the arguments that follow, whenever we are viewing only a portion of the graph we will always assume that the twist occurs elsewhere in the graph.

Various proofs in this paper will go as follows. We will begin with a set of copies of $J_n$ for which the pattern is known, as well as possibly knowing in advance that some vertices are either in $S$, or not in $S$. From there, depending on the variant of domination being considered, and the structure of $J_n$, we will go on to prove that certain vertices either must be in $S$, or must not be in $S$. For these proofs, figures will be provided, using the following convention. The pattern known in advance will be indicated by listing the corresponding weight underneath each copy. Vertices which are known in advance to be in $S$ will be marked with \begin{tikzpicture}[smallgraph,scale=0.8]\inS{b}{1}\end{tikzpicture}, and vertices which are known in advance not to be in $S$ will be marked with \begin{tikzpicture}[smallgraph]\notS{b}{1}\end{tikzpicture}. Then, vertices which are subsequently shown (up to equivalence) to be in $S$ will be marked with \begin{tikzpicture}[smallgraph,scale=0.8]\inSa{b}{1}\end{tikzpicture}, while vertices which are subsequently shown not to be in $S$ will be marked with \begin{tikzpicture}[smallgraph]\notSa{b}{1}\end{tikzpicture}. We demonstrate this convention with a simple example.

\begin{example}Suppose we have four copies $J^1, J^2, J^3, J^4$ which meet the pattern 1111, and that we know that $d^2 \in S$ and $c^1 \not\in S$. This situation is displayed in Figure \ref{fig-example}. Clearly, since $d^2 \in S$ and $J^2$ has weight 1, we know that $a^2 \not\in S$, $b^2 \not\in S$, and $c^2 \not\in S$. These three are marked with \begin{tikzpicture}[smallgraph]\notS{b}{1}\end{tikzpicture} in Figure \ref{fig-example} as no argument was needed to establish they are not in $S$. Then, the only remaining vertex which can dominate $c^2$ is $c^3$, and hence $c^3 \in S$. Since this was argued in the proof, $c^3$ is marked with a \begin{tikzpicture}[smallgraph,scale=0.8]\inSa{b}{1}\end{tikzpicture} in Figure \ref{fig-example}. Also, since $J^3$ has weight 1 the vertices $a^3$, $b^3$, and $d^3$ cannot be in $S$, and so they are marked with \begin{tikzpicture}[smallgraph]\notSa{b}{1}\end{tikzpicture}. Then, the only remaining vertex which can dominate $b^2$ is $b^1$, which we similarly mark in Figure \ref{fig-example}. Finally, the only remaining vertex which can dominate $b^3$ is $b^4$, which we again mark in Figure \ref{fig-example}. Hence, we now know exactly which vertices in $J^1, J^2, J^3, J^4$ are contained in $S$.\label{example}\end{example}

\begin{figure}[h!]\begin{center}\begin{tikzpicture}[smallgraph]\flower{4} \genlab{4} \patt{1}{1} \patt{2}{1} \patt{3}{1} \patt{4}{1} \inS{d}{2} \notS{c}{1} \notS{a}{2} \notS{b}{2} \notS{c}{2} \inSa{b}{1} \inSa{c}{3} \inSa{b}{4} \notSa{a}{1} \notSa{d}{1} \notSa{a}{3} \notSa{b}{3} \notSa{d}{3} \notSa{a}{4} \notSa{c}{4} \notSa{d}{4}\end{tikzpicture}
\caption{The situation described in Example \ref{example}\label{fig-example}.}\end{center}\end{figure}
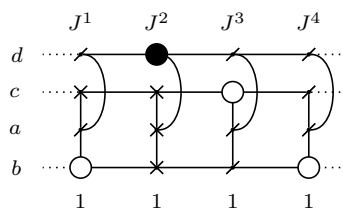

To conclude this section, we note that it is trivial to determine the domination and independent domination numbers for flower snarks.

\begin{lemma}For $n \geq 3$, we have $\gamma(J_n) = i(J_n) = n$.\end{lemma}

\begin{proof}The graph $J_n$ contains $n$ copies of $K_{1,3}$. For each copy $J^i$, the vertex $a^i$ is adjacent only to other vertices in $J^i$, and so any dominating set must contain at least one vertex from each copy. Hence, $n \leq \gamma(J_n) \leq i(J_n)$. Then, it suffices to note that the set $\{a^1, a^2, \hdots, a^n\}$ is an independent dominating set, and hence $i(J_n) \leq n$, leading to the result.\end{proof}

\begin{corollary}For any variant of dominating set considered in this paper, each copy must have weight at least 1.\label{cor-weight1}\end{corollary}

\section{Upper Domination}

In this section, we will determine the upper domination number for flower snarks. We begin by considering what weights are possible for copies in a minimal dominating set.

\begin{lemma}Consider the graph $J_n$ for $n \geq 3$. If $S$ is a minimal dominating set for $J_n$ then the weight of each copy is either 1, 2 or 3.\label{lem-upper-weight}\end{lemma}

\begin{proof}We know from Corollary \ref{cor-weight1} that each copy must have positive weight. Then suppose that a copy $J^i$ has weight 4. It can be easily checked that $S \setminus \{a^i\}$ is also a dominating set, contradicting the assumption that $S$ is minimal.\end{proof}

In the following Lemma, we use the term {\em $i$ depends on $j$} to imply that $S \cap N[i] = \{j\}$. Note that since $S$ is minimal, for every vertex in $S$ there must be at least one other vertex which depends on it.

\begin{lemma}Consider the graph $J_n$ for $n \geq 4$, and a minimal dominating set $S$. If a copy $J^i$ has weight 3 in $S$, then both adjacent copies $J^{i-1}$ and $J^{i+1}$ have weight 1 in $S$.\label{lem-upper-weight3}\end{lemma}

\begin{proof}Suppose it is not the case, that is, $J^i$ has weight 3 and at least one of $J^{i-1}$ and $J^{i+1}$ has weight greater than 1. Since $b^i$, $c^i$ and $d^i$ are all equivalent in this framing, there are only two cases to consider; if $a^i \in S$ and if $a^i \not\in S$.

First, consider the case when $a^i \in S$. Then there are two other vertices from $J^i$ also in $S$. Without loss of generality, suppose $b^i \in S$ and $c^i \in S$. This situation is shown in Figure \ref{fig-upper-131} part(a). Then, since $S$ is minimal, the removal of $a^i$ does not result in a dominating set. This is only possible if $d^i$ depends on $a^i$. Hence, $d^{i-1} \not\in S$ and $d^{i+1} \not\in S$.

Then, consider $b^i$. Since $S$ is minimal, at least one of $b^{i-1}$ and $b^{i+1}$ must depend on $b^i$. Without loss of generality, suppose it is the former. This implies that both $b^{i-1} \not\in S$ and $a^{i-1} \not\in S$. Then, $c^{i-1} \in S$, or else copy $J_{i-1}$ has weight 0 which contradicts Lemma \ref{lem-upper-weight}. Then, since $S$ is minimal, $c^{i+1}$ must depend on $c^i$, which implies that $c^{i+1} \not\in S$ and $a^{i+1} \not\in S$. Hence, from Lemma \ref{lem-upper-weight} copy $J^{i+1}$ also has weight 1, contradicting the initial assumption.

Second, consider the case when $a^i \not\in S$. Then, $S$ contains $b^i$, $c^i$, and $d^i$. Since $S$ is minimal and $a^i$ does not depend on $b^i$, at least one of $b^{i-1}$ or $b^{i+1}$ must depend on $b^i$. Without loss of generality, suppose it is the former. This implies that both $b^{i-1} \not\in S$ and $a^{i-1} \not\in S$. We can make analogous arguments for $c^i$ and $d^i$. Clearly, the dependent vertices can not all be from the same copy, or else that copy has weight 0. Hence, two of the dependent vertices are in one copy and one is in the other; without loss of generality we will assume that $b^{i-1}$ is dependent on $b^i$, $c^{i-1}$ is dependent on $c^i$, and $d^{i+1}$ is dependent on $d^i$. Hence, $J^{i-1}$ has weight 1, and so by assumption, $J^{i+1}$ has weight 2, and $b^{i+1} \in S$ and $c^{i+1} \in S$. But then, by similar arguments, it must be the case that $b^{i+2}$ depends on $b^{i+1}$, and $c^{i+2}$ depends on $c^{i+1}$, implying that $J^{i+2}$ has weight 1, and $d^{i+2} \in S$. This situation is displayed in Figure \ref{fig-upper-131} part (b). Finally, it can be seen from this figure that $S \setminus \{d^i\}$ is a dominating set, contradicting the initial assumption that $S$ is minimal.\end{proof}

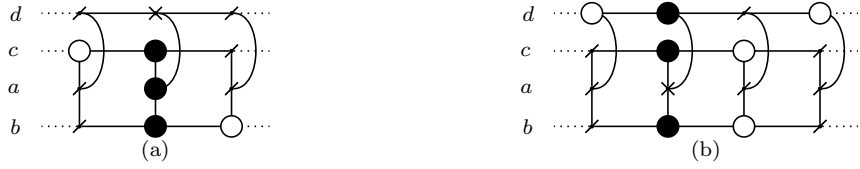
\begin{figure}[h!]\begin{center}\begin{tikzpicture}[smallgraph]\flower{3} \patt{2}{(a)} \inS{a}{2} \inS{b}{2} \inS{c}{2} \inSa{c}{1} \inSa{b}{3} \notSa{a}{1} \notSa{b}{1} \notSa{d}{1} \notS{d}{2} \notSa{a}{3} \notSa{c}{3} \notSa{d}{3}\end{tikzpicture} \hspace*{3cm} \begin{tikzpicture}[smallgraph]\flower{4} \patt{2.5}{(b)} \inS{b}{2} \inS{c}{2} \inS{d}{2} \inSa{d}{1} \inSa{b}{3} \inSa{c}{3} \inSa{d}{4} \notS{a}{2} \notSa{a}{1} \notSa{b}{1} \notSa{c}{1} \notSa{a}{3} \notSa{d}{3} \notSa{a}{4} \notSa{b}{4} \notSa{c}{4}\end{tikzpicture}
\caption{The two situations described in the proof of Lemma \ref{lem-upper-weight3}.\label{fig-upper-131}}\end{center}\end{figure}

Recalling that $w^S_i$ is the number of copies to have weight $i$ in $S$, Lemma \ref{lem-upper-weight3} implies the following.

\begin{corollary}Consider the graph $J_n$ for $n \geq 4$. If $S$ is a minimal dominating set for $J_n$ then $w^S_1 \geq w^S_3$, with equality occurring if and only if $w^S_1 = w^S_3 = \frac{n}{2}$. \label{cor-upper-w13}\end{corollary}

We are now ready to prove the main result of this section.

\begin{theorem}Consider a graph $J_n$ for $n \geq 3$. Then,

$$\Gamma(J_n) = \left\{\begin{aligned}2n\;\;\;, & \mbox{ if } n\mbox{ is even,}\\2n-1, & \mbox{ if } n\mbox{ is odd.}\end{aligned}\right.$$\end{theorem}

\begin{proof}We use the first formulation for upper domination from \cite{upperdom} to confirm that $\Gamma(J_3) = 5$. Then, suppose that $S$ is a minimal dominating set for $J_n$, for $n \geq 4$. From Lemma \ref{lem-upper-weight} we have $w^S_0 = w^S_4 = 0$. Hence, we have $w^S_1 + w^S_2 + w^S_3 = n$, and $w^S_1 + 2w^S_2 + 3w^S_3 = |S|$. Combining these, we obtain $|S| = 2n + w^S_3 - w^S_1$. From Corollary \ref{cor-upper-w13} we know that $w^S_1 \geq w^S_3$, and the inequality is strict if $n$ is odd. Hence, $|S| \leq 2n$ if $n$ is even, and $|S| \leq 2n-1$ if $n$ is odd. Then we just need to obtain the corresponding lower bounds.

Suppose that $n$ is even. It is easy to check that if we repeat the configuration displayed in Figure \ref{fig-minimal} part (a) $n/2$ times, what results is a minimal dominating set with weight $2n$. Hence, if $n$ is even, we have $\Gamma(J_n) \geq 2n$. Then suppose that $n$ is odd. Again, it is easy to check that if we repeat the configuration displayed in Figure \ref{fig-minimal} part (a) $(n-1)/2$ times, and then use the configuration displayed in Figure \ref{fig-minimal} part (b) for the final copy, what results is a minimal dominating set with weight $2n-1$. Hence, if $n$ is odd, $\Gamma(J_n) \geq 2n-1$.
\end{proof}

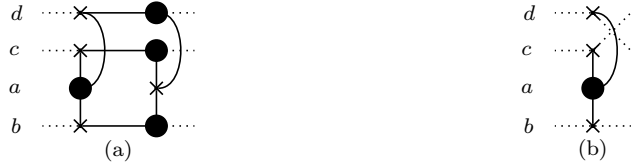
\begin{figure}[h!]\begin{center}\begin{tikzpicture}[smallgraph]\flower{2} \patt{1.5}{(a)} \inS{a}{1} \inS{b}{2} \inS{c}{2} \inS{d}{2} \notS{b}{1} \notS{c}{1} \notS{d}{1} \notS{a}{2}\end{tikzpicture} \hspace{4cm} \begin{tikzpicture}[smallgraph]\flowerend{1} \patt{1}{(b)} \inS{a}{1} \notS{b}{1} \notS{c}{1} \notS{d}{1}\end{tikzpicture}
\caption{The configuration for upper domination which gives the desired lower bound for $\Gamma(J_n)$ for $n \geq 3$. Part (a) can be repeated as many times as necessary. Then, if $n$ is odd, use part (b) to finish. The result is a minimal dominating set with weight $2n$ if $n$ is even, or weight $2n - 1$ if $n$ is odd.\label{fig-minimal}}\end{center}\end{figure}

\section{Upper Bounds by Construction}\label{sec-upper}In the upcoming sections, we will determine lower bounds for the 2-domination, total domination, connected domination, secure domination, and weak Roman domination numbers of flower snarks. To obtain equality, we will require corresponding upper bounds, which we provide here. We leave it as an exercise to the reader to verify that the configurations given here satisfy the requirements of the various kinds of domination, and that they imply the appropriate upper bounds for Theorem \ref{thm-overall}.

For 2-domination, the configuration shown in Figure \ref{fig-2dom} can be repeated as often as necessary, truncating the final time to get the desired size.

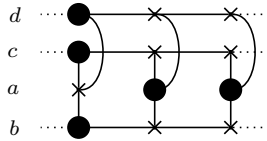
\begin{figure}[h!]\begin{center}\begin{tikzpicture}[smallgraph]\flower{3} \inS{b}{1} \inS{c}{1} \inS{d}{1} \inS{a}{2} \inS{a}{3} \notS{a}{1} \notS{b}{2} \notS{c}{2} \notS{d}{2} \notS{b}{3} \notS{c}{3} \notS{d}{3}\end{tikzpicture}
\caption{The configuration for 2-domination which gives the desired upper bound for $\gamma_2(J_n)$. The result is a 2-dominating set of weight $\lceil\frac{5n}{3}\rceil$ if $n \neq 1\mbox{ mod }3$, or weight $\lceil\frac{5n}{3}\rceil + 1$ if $n = 1\mbox{ mod }3$.\label{fig-2dom}}\end{center}\end{figure}

For total domination, the configuration shown in Figure \ref{fig-total} part (a) can be repeated as often as necessary. Then the configurations in parts (b), (c) and (d) can be used to complete the remaining copies.

\begin{figure}[h!]\begin{center}\begin{tikzpicture}[smallgraph]\flower{4} \patt{2.5}{(a)} \inS{b}{1} \inS{c}{2} \inS{d}{2} \inS{c}{3} \inS{d}{3} \inS{b}{4} \notS{a}{1} \notS{c}{1} \notS{d}{1} \notS{a}{2} \notS{b}{2} \notS{a}{3} \notS{b}{3} \notS{a}{4} \notS{c}{4} \notS{d}{4}\end{tikzpicture} \;\;\;\;\; \begin{tikzpicture}[smallgraph]\flowerend{1} \patt{1}{(b)} \inS{a}{1} \inS{b}{1} \notS{c}{1} \notS{d}{1}\end{tikzpicture} \;\;\;\;\; \begin{tikzpicture}[smallgraph]\flowerend{2} \patt{1.5}{(c)} \inS{a}{1} \inS{b}{1} \inS{a}{2} \inS{b}{2} \notS{c}{1} \notS{d}{1} \notS{c}{2} \notS{d}{2}\end{tikzpicture} \;\;\;\;\; \begin{tikzpicture}[smallgraph]\flowerend{3} \patt{2}{(d)} \inS{b}{1} \inS{a}{2} \inS{c}{2} \inS{d}{2} \inS{b}{3} \notS{a}{1} \notS{c}{1} \notS{d}{1} \notS{b}{2} \notS{a}{3} \notS{c}{3} \notS{d}{3}\end{tikzpicture}
\caption{The configuration for total domination which gives the desired upper bound for $\gamma_t(J_n)$ for $n \geq 3$. Part (a) can be repeated as many times as necessary. If $n = 0\mbox{ mod }4$ this is sufficient. If $n = 1\mbox{ mod }4$, use (b) to finish. If $n = 2\mbox{ mod }4$, use part (c) to finish. If $n = 3\mbox{ mod }4$, use part (d) to finish. The result is a total dominating set of weight $\lceil\frac{3n}{2}\rceil$ if $n \neq 2\mbox{ mod }4$, or weight $\frac{3n}{2} + 1$ if $n = 2\mbox{ mod }4$.\label{fig-total}}\end{center}\end{figure}
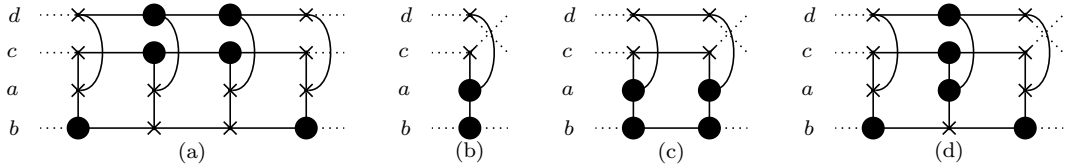

For connected domination, the configuration shown in Figure \ref{fig-connected} can be repeated as often as necessary, truncating the final time to get the desired size. If $n \neq 1\mbox{ mod }4$ the result is connected dominating. If $n = 1\mbox{ mod }4$, then removing $d^1$ and adding $b^1$ gives the desired result.

\begin{figure}[h!]\begin{center}\begin{tikzpicture}[smallgraph]\flower{4} \inS{d}{1} \inS{a}{2} \inS{b}{2} \inS{d}{2} \inS{d}{3} \inS{a}{4} \inS{c}{4} \inS{d}{4} \notS{a}{1} \notS{b}{1} \notS{c}{1} \notS{c}{2} \notS{a}{3} \notS{b}{3} \notS{c}{3} \notS{b}{4}\end{tikzpicture}\caption{The configuration for connected domination which gives the desired upper bound for $\gamma_c(J_n)$ for $n \geq 3$. It may be repeated as many times as necessary, truncating the final time to get the final size. If $n = 1\mbox{ mod }4$, then in the first copy, remove $d^1$ and add $b^1$. The result is a connected dominating set of weight $2n$ if $n$ is even, or $2n - 1$ if $n$ is odd.\label{fig-connected}}\end{center}\end{figure}
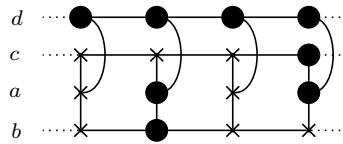

For secure domination, the configuration shown in Figure \ref{fig-secure} part (a) can be repeated as often as necessary. Then the configurations in parts (b), (c), (d), and (e) can be used to complete the remaining copies. Since secure domination is an upper bound for weak Roman domination, these configurations also give an upper bound for weak Roman domination.

\begin{figure}[h!]\begin{center}\begin{tikzpicture}[smallgraph]\flower{4} \patt{2.5}{(a)} \inS{b}{1} \inS{c}{2} \inS{d}{2} \inS{b}{3} \inS{c}{3} \inS{d}{4} \notS{a}{1} \notS{c}{1} \notS{d}{1} \notS{a}{2} \notS{b}{2} \notS{a}{3} \notS{d}{3} \notS{a}{4} \notS{b}{4} \notS{c}{4}\end{tikzpicture} \hfill \begin{tikzpicture}[smallgraph]\flowerend{5} \patt{3}{(c)} \inS{b}{1} \inS{c}{2} \inS{d}{2} \inS{b}{3} \inS{c}{3} \inS{b}{4} \inS{c}{5} \inS{d}{4} \notS{a}{1} \notS{c}{1} \notS{d}{1} \notS{a}{2} \notS{b}{2} \notS{a}{3} \notS{d}{3} \notS{a}{4} \notS{c}{4} \notS{a}{5} \notS{b}{5} \notS{d}{5}\end{tikzpicture} \;\;\;\;\; \begin{tikzpicture}[smallgraph]\flowerend{2} \patt{1.5}{(d)} \inS{b}{1} \inS{c}{1} \inS{c}{2} \inS{d}{2} \notS{a}{1} \notS{d}{1} \notS{a}{2} \notS{b}{2}\end{tikzpicture}\\
\begin{tikzpicture}[smallgraph]\flowerend{4} \patt{2.5}{(b)} \inS{b}{1} \inS{c}{2} \inS{d}{2} \inS{b}{3} \inS{c}{3} \inS{c}{4} \inS{d}{4} \notS{a}{1} \notS{c}{1} \notS{d}{1} \notS{a}{2} \notS{b}{2} \notS{a}{3} \notS{d}{3} \notS{a}{4}\end{tikzpicture} \hfill \begin{tikzpicture}[smallgraph]\flowerend{7} \patt{4}{(e)} \inS{b}{1} \inS{c}{2} \inS{d}{2} \inS{a}{3} \inS{b}{3} \inS{a}{4} \inS{a}{5} \inS{c}{5} \inS{b}{6} \inS{d}{6} \inS{c}{7} \notS{a}{1} \notS{c}{1} \notS{d}{1} \notS{a}{2} \notS{b}{2} \notS{c}{3} \notS{d}{3} \notS{b}{4} \notS{c}{4} \notS{d}{4} \notS{b}{5} \notS{d}{5} \notS{a}{6} \notS{c}{6} \notS{a}{7} \notS{b}{7} \notS{d}{7} \node[white] at (8.25,1) {};\end{tikzpicture}
\caption{The configuration for secure domination which gives the desired upper bound for $\gamma_s(J_n)$ for $n \geq 4$. Part (a) can be repeated as many times as necessary. If $n = 0\mbox{ mod }4$, use part (b) to finish. If $n = 1\mbox{ mod }4$, use part (c) to finish. If $n = 2\mbox{ mod }4$, use part (d) to finish. If $n = 3\mbox{ mod }4$, use part (e) to finish. In all cases, the result is a secure dominating set of weight $\lceil\frac{3n+1}{2}\rceil$.\label{fig-secure}}\end{center}\end{figure}
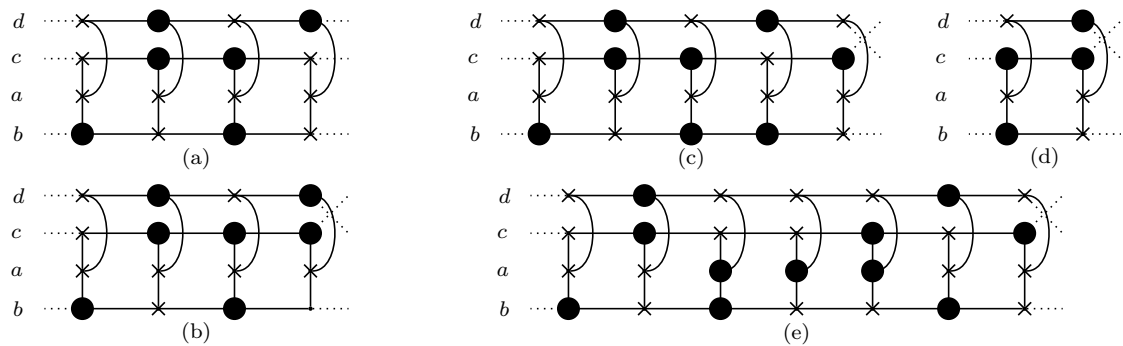

\section{2-domination}

In this section, we will determine the 2-domination numbers for flower snarks.

\begin{lemma}Consider the graph $J_n$ for $n \geq 3$, and a 2-dominating set $S$. If the copy $J^i$ has weight 1 in $S$, then $a^i \in S$.\label{lem-2dom-weight1}\end{lemma}

\begin{proof}Recall that $a^i$ is adjacent only to other vertices in $J^i$. From the definition of 2-domination, if $a^i \not\in S$ then it must have at least two neighbours in $S$. Since $J^i$ has weight 1, this is impossible, and so $a^i \in S$.\end{proof}

%\begin{theorem}Consider the graph $J_n$ for $n \geq 3$, and a 2-dominating set $S$. Then $|S| \geq \lceil\frac{5n}{3}\rceil$.\label{thm-2dom-5n3}\end{theorem}
\begin{theorem}Consider the graph $J_n$ for $n \geq 3$, and a 2-dominating set $S$. Then any copy with weight 1 in $S$ has a neighbouring copy with weight 3 or 4 in $S$.\label{thm-2dom-13}\end{theorem}

\begin{proof}Without loss of generality, suppose that $J^2$ has weight 1. Then, from Lemma \ref{lem-2dom-weight1} we have $a^2 \in S$. Then, suppose that its neighbours $J^1$ and $J^3$ both have weight less than 3. At this stage, vertices $b^2$, $c^2$ and $d^2$ have only one neighbour in $S$, so they each need at least one more. Without loss of generality, suppose that $b^1 \in S$. Then, consider the case when $c^1 \in S$. Since $J^1$ has weight less than 3, this implies that $a^1 \not\in S$ and $d^1 \not\in S$. However, it is then impossible for $d^1$ to have two neighbours in $S$. This situation is displayed in part (a) of Figure \ref{fig-2dom-13}.

\begin{figure}[h!]\begin{center}\begin{tikzpicture}[smallgraph]\flower{2} \genlab{2} \inS{a}{2} \notS{b}{2} \notS{c}{2} \notS{d}{2} \inSa{b}{1} \inSa{c}{1} \notSa{a}{1} \notSa{d}{1}\node[white] [label=right:{(a)}] at (3,1.25) {}; \end{tikzpicture} \hspace*{2.5cm} \begin{tikzpicture}[smallgraph]\flower{3} \genlab{3} \inS{a}{2} \notS{b}{2} \notS{c}{2} \notS{d}{2} \inSa{a}{1} \inSa{b}{1} \notSa{c}{1} \notSa{d}{1} \inSa{c}{3} \inSa{d}{3} \notSa{a}{3} \notSa{b}{3}\node[white] [label=right:{(b)}] at (4,1.25) {}; \end{tikzpicture}
\caption{The two situations described in the proof of Theorem \ref{thm-2dom-13}. In part (a), $d^1$ cannot have two neighbours in $S$. In part (b), copy $b^3$ cannot have two neighbours in $S$.\label{fig-2dom-13}}\end{center}\end{figure}
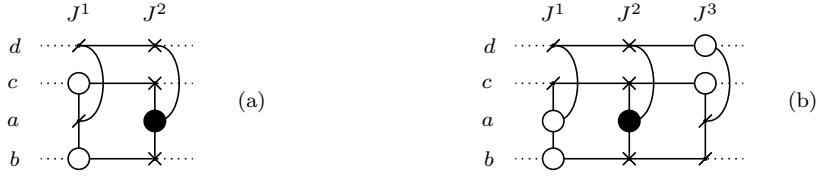

Hence, we must have $c^1 \not\in S$. Then, $c^1$ must have two neighbours in $S$, which implies that $a^1 \in S$. Since $J^1$ has weight less than 3, this implies that $d^1 \not\in S$. Then, in order for $c^2$ and $d^2$ to have two neighbours in $S$, we must have $c^3 \in S$ and $d^3 \in S$, respectively. Since $J^3$ has weight less than 3, this implies that $a^3 \not\in S$ and $b^3 \not\in S$. However, it is then impossible for $b^3$ to have two neighbours in $S$, completing the proof. This situation is displayed in part (b) of Figure \ref{fig-2dom-13}.\end{proof}

We are now ready to prove the main result of this theorem.

\begin{theorem}\label{thm-2dom}Consider the graph $J_n$ for $n \geq 3$. Then,

$$\gamma_2(J_n) = \left\{\begin{aligned}\left\lceil\frac{5n}{3}\right\rceil\;, & \mbox{ if } n \neq 1\mbox{ mod }3,\\\frac{5n+4}{3}, & \mbox{ if } n = 1\mbox{ mod }3,\end{aligned}\right.$$\end{theorem}

\begin{proof}The upper bound was established in Section \ref{sec-upper}. Then, from Corollary \ref{cor-weight1} we know that each copy has weight at least 1. This, combined with Theorem \ref{thm-2dom-13}, implies that any three set of consecutive copies have weight at least 5. Hence we have $\gamma_2(S) \geq \lceil\frac{5n}{3}\rceil$. Hence, Theorem \ref{thm-2dom} is true for $n \neq 1\mbox{ mod }3$.

Suppose that $n = 1\mbox{ mod }3$, and we have a 2-dominating set $S$ such that $|S| < \frac{5n+4}{3}$. Hence, we must have $|S| = \frac{5n+1}{3}$. For each copy $J^i$, denote by $w^3(i)$ the combined weights of $J^i$, $J^{i+1}$, and $J^{i+2}$, where the superscripts are taken modulo $n$. Recall that $w^3(i) \geq 5$. Hence there must be a single value $k$ such that $w^3(k) = 6$ and $w^3(j) = 5$ for all $j \neq k$. Consider the copies $J^k$, $J^{k+1}$, $J^{k+2}$. We will consider the possible patterns they could have, and show that each is impossible. Since $w^3(k)$ has weight 6, the only possible patterns (up to symmetry) for copies $J^k$, $J^{k+1}$ and $J^{k+2}$ are 123, 132, 222, and 312.

Suppose that $J^k$, $J^{k+1}$ and $J^{k+2}$ meet either the pattern 123 or the pattern 132. Since $w^3(k+1) = 5$, this implies that $J^{k+3}$ has weight 0, which contradicts Corollary \ref{cor-weight1}.

Suppose next that $J^k$, $J^{k+1}$ and $J^{k+2}$ meet the pattern 222. Since $w^3(k+1) = w^3(k+2) = 5$, this implies that $J^{k+3}$ has weight 1, and $J^{k+4}$ has weight 2. However, this contradicts Theorem \ref{thm-2dom-13} since $J^{k+3}$ has no neighbour with weight 3 or 4. Hence, this is impossible.

Finally, suppose that $J^k$, $J^{k+1}$ and $J^{k+2}$ meet pattern 312. Since $w^3(k+1) = w^3(k+2) = w^3(k+3) = 5$, this implies that $J^{k+3}$ has weight 2, $J^{k+4}$ has weight 1, and $J^{k+5}$ has weight 2. Again, this contradicts Theorem \ref{thm-2dom-13} since $J^{k+4}$ has no neighbour with weight 3 or 4. Hence, all cases are impossible, and so $|S| \geq \frac{5n+4}{3}$, completing the proof.\end{proof}

\section{Total Domination}

In this section, we will determine the total domination numbers for flower snarks. We begin by identifying three patterns which cannot occur in total dominating sets of $J_n$.

\begin{theorem}Consider the graph $J_n$ for $n \geq 3$, and a total dominating set $S$. Then $S$ does not contain the patterns 111, 1121, or 12121.\label{thm-total-patterns}\end{theorem}

\begin{proof}From Corollary \ref{cor-weight1}, each copy of $J_n$ has weight at least 1 in $S$. Also, if a copy $J^i$ has weight 1, then $a^i \not\in S$ because otherwise $a^i$ itself is not dominated.

Suppose that $S$ has the pattern 111. Without loss of generality, suppose the three copies meeting this pattern are $J^1, J^2, J^3$. As indicated above, $a^2 \not\in S$. Regardless of which vertex from $J^2$ is in $S$, it does not dominate any of $b^2, c^2, d^2$. Also, any vertex from $J^1$ dominates at most one vertex from $J^2$, and likewise for any vertex from $J^3$. Hence there is at least one vertex in $J^2$ which is not dominated, contradicting the assumption that $S$ is a total dominating set.

Then, suppose that $S$ has the pattern 1121. Without loss of generality, suppose the four copies meeting this pattern are $J^1, J^2, J^3, J^4$. Using an equivalent argument to above, it must be the case that $b^2, c^2, d^2$ are dominated by one vertex from $J^1$ and two vertices from $J^3$. Hence, $a^3 \not\in S$. This means none of the vertices from $J^3$ dominate any of $b^3, c^3, d^3$. Also, any vertex from $J^2$ dominates at most one vertex from $J^3$, and likewise for any vertex from $J^4$. Hence there is at least one vertex in $J^3$ which is not dominated, contradicting the assumption that $S$ is a total dominating set.

Finally, suppose that $S$ has the pattern 12121. Without loss of generality, suppose the five copies meeting this pattern are $J^1, \hdots, J^5$. Since $J^1$ and $J^3$ are both weight 1, they can collectively dominate at most two vertices from $J^2$. Hence, $a^2 \in S$, and by an equivalent argument, $a^4 \in S$. Note that none of the vertices from $J^3$ can dominate any of $b^3$, $c^3$ or $d^3$. Also, the vertices from $J^2$ dominate at most one vertex from $J^3$, and likewise for the vertices from $J^4$. Hence there is at least one vertex in $J^3$ which is not dominated, contradicting the assumption that $S$ is a total dominating set.\end{proof}

An immediate observation arising from Corollary \ref{cor-weight1} and Theorem \ref{thm-total-patterns} is that any four consecutive copies must collectively have weight at least 6 in $S$, leading to the following corollary.

\begin{corollary}For $n \geq 4$, $\gamma_t(J_n) \geq \left\lceil\frac{3n}{2}\right\rceil.$\label{cor-total}\end{corollary}

We are now ready to prove the main result of this section.

\begin{theorem}Consider the graph $J_n$ for $n \geq 3$. Then,

$$\gamma_t(J_n) = \left\{\begin{aligned}\left\lceil\frac{3n}{2}\right\rceil\;, & \mbox{ if } n \neq 2\mbox{ mod }4,\\\frac{3n}{2} + 1, & \mbox{ if } n = 2\mbox{ mod }4,\end{aligned}\right.$$\label{thm-total}\end{theorem}

\begin{proof}The upper bound was established in Section \ref{sec-upper}. We use the formulation for total domination from \cite{burger} to confirm that $\gamma_t(J_3) = 5$. Then, suppose there is some value $k \geq 4$ such that Theorem \ref{thm-total} is false. If $k \neq 2\mbox{ mod }4$ then we have $\gamma_t(J_k) \leq \left\lceil\frac{3k}{2}\right\rceil - 1$, contradicting Corollary \ref{cor-total}. Hence, we must have $k = 2\mbox{ mod }4$. Hence, $\gamma_t(J_k) \leq \frac{3k}{2}$, and from Corollary \ref{cor-total} this implies that $\gamma_t(J_k) = \frac{3k}{2}$.

Suppose that we have a total dominating set $S$ with weight $\frac{3k}{2}$. Note that this implies that {\em every} set of four consecutive copies of $J_n$ has weight 6 in $S$; call this property 1. It is clear that property 1 implies that no copy has weight 4 in $S$. If a copy has weight 3 in $S$, then in order to satisfy property 1, the next three copies must have weight 1, which contradicts Theorem \ref{thm-total-patterns}. Hence, every copy has weight 1 or 2 in $S$; call this property 2. It can be easily checked that there are only two ways to satisfy properties 1 and 2 without contradicting Theorem \ref{thm-total-patterns}; either $S$ contains the repeated pattern $1212 \hdots 12$, or $S$ contains the repeated pattern $11221122 \hdots 1122$. In the former case, $S$ contains the pattern 12121 which by Theorem \ref{thm-total-patterns} is impossible. Hence, we must have the latter case. Also, the latter case is impossible because $k = 2\mbox{ mod }4$. Hence, all cases are impossible, completing the proof.\end{proof}

\section{Connected Domination}

In this section, we will determine the connected domination numbers for flower snarks. The following three remarks are obvious, but we list them here to aid the readability of two upcoming proofs.

\begin{remark}Suppose that $S$ is a connected dominating set for a graph $G$. If a new graph $H$ is created by either (a) adding an edge between two non-adjacent vertices in $G$, or (b) deleting a vertex $v \not\in S$ from $G$, then $S$ is also a connected dominating set for $H$.\label{rem-remove}\end{remark}

\begin{remark}Suppose that $S$ is a connected dominating set for a graph $G$, and that $G$ contains a degree 1 vertex $v \in S$ whose neighbour has degree larger than 1. If a new graph $H$ is created by deleting $v$ from $G$, then $S \setminus \{v\}$ is a connected dominating set for $H$.\label{rem-path}\end{remark}

\begin{remark}Suppose that $S$ is a connected dominating set for a graph $G$, and that $G$ contains a triangle $uvw$ such that $v$ is degree 2, and $v \in S$. If a new graph $H$ is created by deleting $v$ from $G$, then $S \setminus \{v\}$ is a connected dominating set for $H$.\label{rem-triangle}\end{remark}

In the following, we define $V^S(J^i) := S \cap (a^i, b^i, c^i, d^i)$, that is, $V^S(J^i)$ is the set of vertices from $J^i$ that are contained in $S$.

\begin{theorem}Consider the graph $J_n$ for $n \geq 4$, and let $S$ be a connected dominating set for $J_n$. Suppose that there is a copy $J^i$ such that either $a^i \not\in S$, or $J^i$ has weight 2 in $S$. Then $\overline{S} := S \setminus V^S(J^i)$ is a connected dominating set for $J_{n-1}$.\label{thm-connected-remove}\end{theorem}

\begin{proof}Consider first the situation when $a^i \not\in S$, and consider the graph $J_{n-1}$. One may think of $J_{n-1}$ as being constructed by starting with $J_n$ and then ``smoothing out" copy $J^i$, in the following sense. First, we add edges connecting $b^{i-1}$ to $b^{i+1}$, $c^{i-1}$ to $c^{i+1}$ and $d^{i-1}$ to $d^{i+1}$, and then we delete the vertex set $V(J^i)$. Suppose that we are midway through this process, having added all three edges, and having deleted $a^i$. Call this intermediate graph $G$, and note from Remark \ref{rem-remove} that $S$ is a connected dominating set for $G$. Then, note that in $G$, $b^i$ is a degree 2 vertex and is part of a triangle $b^{i-1}b^ib^{i+1}$. Now, suppose that we delete $b^i$ from $G$, to obtain a new intermediate graph $G_2$. If $b^i \not\in S$ then from Remark \ref{rem-remove} we can see that $S$ is a connected dominating set for $G_2$. If $b^i \in S$ then from Remark \ref{rem-triangle} we can see that $S \setminus \{b^i\}$ is a connected dominating set for $G_2$. Applying analogous arguments for $c^i$ and $d^i$, we obtain the result. The graphs $G$ and $G_2$ are displayed in Figure \ref{fig-connected-remove}

We next consider the situation when $J^i$ has weight 2 in $S$. If $a^i \not\in S$ then the previous paragraph applies. If $a^i \in S$ then there must be one more vertex from $J^i$ also in $S$; without loss of generality, suppose that $b^i \in S$. Again, we construct $J_{n-1}$ by adding edges to and deleting vertices from $J_n$. Suppose that we are midway through this process, having added all three edges, and having deleted $c^i$ and $d^i$. Call this intermediate graph $G_3$. From Remark \ref{rem-remove} we know that $S$ is a connected dominating set for $G_3$. Note that in $G_3$, $a^i$ is a degree 1 vertex. Remark \ref{rem-path} implies that we can obtain a second intermediate graph, $G_4$, by deleting $a^i$, and that $S \setminus \{a^i\}$ is a connected dominating set for $G_4$. Finally, note that in $G_4$, $b^i$ is degree 2 vertex and is part of a triangle $b^{i-1}b^ib^{i+1}$. Hence, Remark \ref{rem-triangle} implies the result. The graphs $G_3$ and $G_4$ are displayed in Figure \ref{fig-connected-remove}.\end{proof}

\begin{figure}[h!]\begin{center}\begin{tikzpicture}[smallgraph,scale=0.85]\node (a1) at (1,1) {}; \node (b1) at (1,0.5) {}; \node (c1) at (1,1.5) {}; \node (d1) at (1,2) {}; \node (b2) at (2,0.5) {}; \node (c2) at (2,1.5) {}; \node (d2) at (2,2) {}; \node (a3) at (3,1) {}; \node (b3) at (3,0.5) {}; \node (c3) at (3,1.5) {}; \node (d3) at (3,2) {}; \draw (b1) to (b3); \draw (c1) to (c3); \draw (d1) to (d3); \draw (b1) to (c1); \draw (b3) to (c3);
\draw[bend right=30] (b1) to (b3); \draw[bend left=30] (c1) to (c3); \draw[bend left=30] (d1) to (d3); \draw[bend right=90] (a1) to (d1); \draw[bend right=90] (a3) to (d3);
\draw[dotted] (b1) to (0.5,0.5); \draw[dotted] (c1) to (0.5,1.5); \draw[dotted] (d1) to (0.5,2); \draw[dotted] (b3) to (3.5,0.5); \draw[dotted] (c3) to (3.5,1.5); \draw[dotted] (d3) to (3.5,2); \node[white] [label=right:{$a$}] at (0,1) {};\node[white] [label=right:{$b$}] at (0,0.5) {};  \node[white] [label=right:{$c$}] at (0,1.5) {}; \node[white] [label=right:{$d$}] at (0,2) {};
\node[white] [label=above:{$G$}] at (2,-0.5) {};\end{tikzpicture} \;\;\;
\begin{tikzpicture}[smallgraph,scale=0.85]\node (a1) at (1,1) {}; \node (b1) at (1,0.5) {}; \node (c1) at (1,1.5) {}; \node (d1) at (1,2) {}; \node (c2) at (2,1.5) {}; \node (d2) at (2,2) {}; \node (a3) at (3,1) {}; \node (b3) at (3,0.5) {}; \node (c3) at (3,1.5) {}; \node (d3) at (3,2) {}; \draw (c1) to (c3); \draw (d1) to (d3); \draw (b1) to (c1); \draw (b3) to (c3);
\draw[bend right=30] (b1) to (b3); \draw[bend left=30] (c1) to (c3); \draw[bend left=30] (d1) to (d3); \draw[bend right=90] (a1) to (d1); \draw[bend right=90] (a3) to (d3);
\draw[dotted] (b1) to (0.5,0.5); \draw[dotted] (c1) to (0.5,1.5); \draw[dotted] (d1) to (0.5,2); \draw[dotted] (b3) to (3.5,0.5); \draw[dotted] (c3) to (3.5,1.5); \draw[dotted] (d3) to (3.5,2); \node[white] [label=right:{$a$}] at (0,1) {};\node[white] [label=right:{$b$}] at (0,0.5) {};  \node[white] [label=right:{$c$}] at (0,1.5) {}; \node[white] [label=right:{$d$}] at (0,2) {};
\node[white] [label=above:{$G_2$}] at (2,-0.5) {};\end{tikzpicture} \;\;\;
\begin{tikzpicture}[smallgraph,scale=0.85]\node (a1) at (1,1) {}; \node (b1) at (1,0.5) {}; \node (c1) at (1,1.5) {}; \node (d1) at (1,2) {}; \node (a2) at (2,1) {}; \node (b2) at (2,0.5) {}; \node (a3) at (3,1) {}; \node (b3) at (3,0.5) {}; \node (c3) at (3,1.5) {}; \node (d3) at (3,2) {}; \draw (b1) to (b3); \draw (b1) to (c1); \draw (b3) to (c3); \draw (a2) to (b2); \inS{a}{2} \inS{b}{2}
\draw[bend right=30] (b1) to (b3); \draw[bend left=30] (c1) to (c3); \draw[bend left=30] (d1) to (d3); \draw[bend right=90] (a1) to (d1); \draw[bend right=90] (a3) to (d3);
\draw[dotted] (b1) to (0.5,0.5); \draw[dotted] (c1) to (0.5,1.5); \draw[dotted] (d1) to (0.5,2); \draw[dotted] (b3) to (3.5,0.5); \draw[dotted] (c3) to (3.5,1.5); \draw[dotted] (d3) to (3.5,2); \node[white] [label=right:{$a$}] at (0,1) {};\node[white] [label=right:{$b$}] at (0,0.5) {};  \node[white] [label=right:{$c$}] at (0,1.5) {}; \node[white] [label=right:{$d$}] at (0,2) {};
\node[white] [label=above:{$G_3$}] at (2,-0.5) {};\end{tikzpicture} \;\;\;
\begin{tikzpicture}[smallgraph,scale=0.85]\node (a1) at (1,1) {}; \node (b1) at (1,0.5) {}; \node (c1) at (1,1.5) {}; \node (d1) at (1,2) {}; \node (b2) at (2,0.5) {}; \node (a3) at (3,1) {}; \node (b3) at (3,0.5) {}; \node (c3) at (3,1.5) {}; \node (d3) at (3,2) {}; \draw (b1) to (b3); \draw (b1) to (c1); \draw (b3) to (c3); \inS{b}{2}
\draw[bend right=30] (b1) to (b3); \draw[bend left=30] (c1) to (c3); \draw[bend left=30] (d1) to (d3); \draw[bend right=90] (a1) to (d1); \draw[bend right=90] (a3) to (d3);
\draw[dotted] (b1) to (0.5,0.5); \draw[dotted] (c1) to (0.5,1.5); \draw[dotted] (d1) to (0.5,2); \draw[dotted] (b3) to (3.5,0.5); \draw[dotted] (c3) to (3.5,1.5); \draw[dotted] (d3) to (3.5,2); \node[white] [label=right:{$a$}] at (0,1) {};\node[white] [label=right:{$b$}] at (0,0.5) {};  \node[white] [label=right:{$c$}] at (0,1.5) {}; \node[white] [label=right:{$d$}] at (0,2) {};
\node[white] [label=above:{$G_4$}] at (2,-0.5) {};\end{tikzpicture}
\caption{Sections of the four intermediate graphs $G$, $G_2$, $G_3$, $G_4$ from the proof of Theorem \ref{thm-connected-remove}.\label{fig-connected-remove}}\end{center}\end{figure}
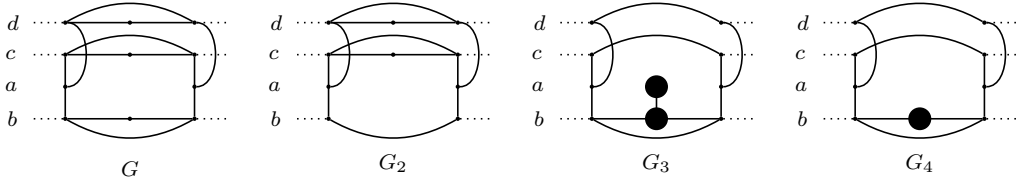

We also make use of the concept of ``smoothing out" a copy in the following lemma.

\begin{lemma}Consider the graph $J_n$ for $n \geq 4$, and let $S$ be a connected dominating set for $J_n$. Suppose that there is a copy $J^i$ with weight 4 in $S$. Then for any positive integer $k \leq n-3$, we have that $\overline{S} := S \setminus (V^S(J^{i+1}) \cup \hdots \cup V^S(J^{i+k}))$ is a connected dominating set for $J_{n-k}$.\label{lem-connected-4}\end{lemma}

\begin{proof}Suppose that we delete from $J_n$ all vertices from each of the copies $J^{i+1}, \hdots, J^{i+k}$. Call this intermediate graph $G$. Then, we add edges connecting $b^i$ to $b^{i+k+1}$, $c^i$ to $c^{i+k+1}$ and $d^i$ to $d^{i+k+1}$, and call the resulting graph $G_2$. Since we can always relabel the vertices so that the twist occurs elsewhere in the graph, it can then be seen that $G_2$ is isomorphic to $J_{n-k}$. Both $G$ and $G_2$ are displayed in Figure \ref{fig-connected-GG2}. Since $b^i \in S$, $c^i \in S$, and $d^i \in S$, it is clear that vertices $b^{i+k+1}$, $c^{i+k+1}$ and $d^{i+k+1}$ are dominated in $G_2$. Also, from Corollary \ref{cor-weight1} we know that $J^{i+k+1}$ has weight at least 1, and so $a^{i+k+1}$ is dominated. Hence, $\overline{S}$ is a dominating set for $G_2$. Next, we need to show that $\overline{S}$ is a connected dominating set for $G_2$. Recall the intermediate graph $G$. It is clear that the subgraph of $G$ induced by $\overline{S}$ is either connected, or has exactly two connected components. In the former case, the subgraph of $G_2$ induced by $\overline{S}$ is also connected. In the latter case, it is clear that there must be an edge $vw$ in $G_2$ such that $v \in J^i$, $w \in J^{i+k+1}$ and both $v,w$ are contained in $S$, and hence the subgraph of $G_2$ induced by $\overline{S}$ is connected. Either way, we obtain the desired result.\end{proof}

\begin{figure}[h!]\begin{center}\begin{tikzpicture}[smallgraph,scale=0.8] \node (a1) at (1,1) {}; \node (b1) at (1,0.5) {}; \node (c1) at (1,1.5) {}; \node (d1) at (1,2) {}; \node (a2) at (2,1) {}; \node (b2) at (2,0.5) {}; \node (c2) at (2,1.5) {}; \node (d2) at (2,2) {}; \node (a5) at (5,1) {}; \node (b5) at (5,0.5) {}; \node (c5) at (5,1.5) {}; \node (d5) at (5,2) {}; \node (a6) at (6,1) {}; \node (b6) at (6,0.5) {}; \node (c6) at (6,1.5) {}; \node (d6) at (6,2) {}; \draw (b1) to (c1); \draw (b2) to (c2); \draw (b5) to (c5); \draw (b6) to (c6); \draw (b1) to (b2); \draw (c1) to (c2); \draw (b5) to (b6); \draw (c5) to (c6); \draw (d1) to (d2); \draw (d5) to (d6); \inS{a}{2} \inS{b}{2} \inS{c}{2} \inS{d}{2}
\node[white] at (3,0) {};
\draw[bend right=90] (a1) to (d1);\draw[bend right=90] (a2) to (d2);\draw[bend right=90] (a5) to (d5);\draw[bend right=90] (a6) to (d6);
\draw[dotted] (b1) to (0.5,0.5); \draw[dotted] (c1) to (0.5,1.5); \draw[dotted] (d1) to (0.5,2); \draw[dotted] (b6) to (6.5,0.5); \draw[dotted] (c6) to (6.5,1.5); \draw[dotted] (d6) to (6.5,2); \node[white] [label=right:{$a$}] at (0,1) {};\node[white] [label=right:{$b$}] at (0,0.5) {};  \node[white] [label=right:{$c$}] at (0,1.5) {}; \node[white] [label=right:{$d$}] at (0,2) {};
\node[white] [label=above:{$J^i \;\;\;\;\;\;\;\;\; \hdots \;\;\;\;\;\;\;\;\; J^{i+k+1}$}] at (3.8,0.7) {};
\node[white] [label=above:{$G$}] at (3.5,-0.5) {};\end{tikzpicture} \hspace*{2.4cm}
\begin{tikzpicture}[smallgraph,scale=0.8] \node (a1) at (1,1) {}; \node (b1) at (1,0.5) {}; \node (c1) at (1,1.5) {}; \node (d1) at (1,2) {}; \node (a2) at (2,1) {}; \node (b2) at (2,0.5) {}; \node (c2) at (2,1.5) {}; \node (d2) at (2,2) {}; \node (a5) at (5,1) {}; \node (b5) at (5,0.5) {}; \node (c5) at (5,1.5) {}; \node (d5) at (5,2) {}; \node (a6) at (6,1) {}; \node (b6) at (6,0.5) {}; \node (c6) at (6,1.5) {}; \node (d6) at (6,2) {}; \draw (b1) to (c1); \draw (b2) to (c2); \draw (b5) to (c5); \draw (b6) to (c6); \draw (b1) to (b2); \draw (c1) to (c2); \draw (b5) to (b6); \draw (c5) to (c6); \draw (d1) to (d2); \draw (d5) to (d6); \draw[bend right=20] (b2) to (b5); \draw[bend left=20] (c2) to (c5); \draw[bend left=20] (d2) to (d5);
\inS{a}{2} \inS{b}{2} \inS{c}{2} \inS{d}{2}
\draw[bend right=90] (a1) to (d1);\draw[bend right=90] (a2) to (d2);\draw[bend right=90] (a5) to (d5);\draw[bend right=90] (a6) to (d6);
\draw[dotted] (b1) to (0.5,0.5); \draw[dotted] (c1) to (0.5,1.5); \draw[dotted] (d1) to (0.5,2); \draw[dotted] (b6) to (6.5,0.5); \draw[dotted] (c6) to (6.5,1.5); \draw[dotted] (d6) to (6.5,2); \node[white] [label=right:{$a$}] at (0,1) {};\node[white] [label=right:{$b$}] at (0,0.5) {};  \node[white] [label=right:{$c$}] at (0,1.5) {}; \node[white] [label=right:{$d$}] at (0,2) {};
\node[white] [label=above:{$J^i \;\;\;\;\;\;\;\;\; \hdots \;\;\;\;\;\;\;\;\; J^{i+k+1}$}] at (3.8,0.7) {};
\node[white] [label=above:{$G_2$}] at (3.5,-0.5) {};\end{tikzpicture}
\caption{Sections of the two intermediate graphs $G$ and $G_2$ from the proof of Lemma \ref{lem-connected-4}.\label{fig-connected-GG2}}\end{center}\end{figure}
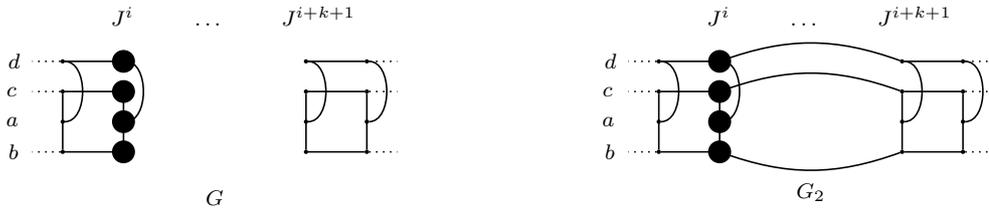

\begin{lemma}Suppose that there is an odd value $n \geq 5$ such that $\gamma_c(J_{n-1}) = 2n-2$. If $S$ is a connected dominating set for $J_n$ such that $|S| = 2n-1$, then $w^S_4$ is even.\label{lem-connected-odd}\end{lemma}

\begin{proof}Suppose that $w^S_2 > 0$. Then there is a copy $J^i$ with weight 2. From Theorem \ref{thm-connected-remove}, we can obtain a connected dominating set for $J_{n-1}$ of cardinality $2n-3$, contradicting the initial assumption. Hence, $w^S_2 = 0$. Then, we have $w^S_1 + w^S_3 + w^S_4 = n$, and $w^S_1 + 3w^S_3 + 4w^S_4 = 2n-1$. Combining these two, we obtain $3w^S_4 = n-1-2w^S_3$. Since $n$ is odd, this implies that $w^S_4$ is even.\end{proof}

\begin{theorem}Consider the graph $J_n$ for $n \geq 3$, and suppose that $S$ is a connected dominating set. Then $S$ does not contain the patterns 111 or 1312131.\label{thm-connected-patterns}\end{theorem}

\begin{proof}Any connected dominating set with $|S| \geq 2$ is also a total dominating set, and hence the result for pattern 111 follows immediately from Theorem \ref{thm-total-patterns}.

Next, consider the case when $S$ contains the pattern 1312131. Without loss of generality, suppose the copies meeting this pattern are $J^1, \hdots, J^7$. Suppose that $a^4 \not\in S$. Then, without loss of generality, suppose that $b^4 \in S$, $c^4 \in S$, and $d^4 \not\in S$. This situation is displayed in part (a) of Figure \ref{fig-connected-patterns}. Since $S$ is dominating, it must contain at least one of $d^3$ and $d^5$. Since both $J^3$ and $J^5$ have weight 1, both options are equivalent, so without loss of generality we will assume that $d^5 \in S$. Then $a^5 \not \in S$, $b^5 \not\in S$, and $c^5 \not\in S$. In order for $S$ to be connected, we must have $b^3 \in S$ and $c^3 \in S$, but this is impossible since $J^3$ has weight 1. Hence, it must be the case that $a^4 \in S$.

Again, without loss of generality, suppose that $b^4 \in S$, $c^4 \not\in S$ and $d^4 \not\in S$. This situation is displayed in part (b) of Figure \ref{fig-connected-patterns}. Since $S$ is connected, it must contain at least one $b^3$ or $b^5$. As in the previous paragraph, without loss of generality we will assume that $b^5 \in S$. Then, in order for $S$ to be dominating, we must have $c^6 \in S$ and $d^6 \in S$. Now, suppose that $b^6 \in S$. Then, since $S$ is connected, it must contain both $c^7$ and $d^7$, but this is impossible since $J^7$ has weight 1. Hence, it must be the case that $b^6 \not\in S$, and hence $a^6 \in S$. Then, since $S$ is connected, we must have $b^3 \in S$ and $b^2 \in S$. In order for $S$ to be dominating, we must have $c^2 \in S$ and $d^2 \in S$, and since $J^2$ has weight 3 this implies that $a^2 \not\in S$.  Finally, to ensure $S$ is connected, it must contain each of $b^1$, $c^1$ and $d^1$, but this is impossible since $J^1$ has weight 1.\end{proof}

\begin{figure}[h!]\begin{center}\begin{tikzpicture}[smallgraph]\flower{3} \lab{1}{$J^3$} \lab{2}{$J^4$} \lab{3}{$J^5$} \patt{1}{1} \patt{2}{2} \patt{3}{1} \inS{b}{2} \inS{c}{2} \notS{a}{2} \notS{d}{2} \inSa{b}{1} \inSa{c}{1} \inSa{d}{3} \notSa{a}{3} \notSa{b}{3} \notSa{c}{3} \node[white] [label=right:{(a)}] at (4,1.25) {}; \end{tikzpicture} \hfill \begin{tikzpicture}[smallgraph]\flower{7} \genlab{7} \patt{1}{1} \patt{2}{3} \patt{3}{1} \patt{4}{2} \patt{5}{1} \patt{6}{3} \patt{7}{1} \inS{a}{4} \inS{b}{4} \notS{c}{4} \notS{d}{4} \inSa{b}{1} \inSa{c}{1} \inSa{d}{1} \inSa{b}{2} \inSa{c}{2} \inSa{d}{2} \inSa{b}{3} \inSa{b}{5} \inSa{a}{6} \inSa{c}{6} \inSa{d}{6} \notSa{a}{1} \notSa{a}{2} \notSa{a}{3} \notSa{c}{3} \notSa{d}{3} \notSa{a}{5} \notSa{c}{5} \notSa{d}{5} \notSa{b}{6} \node[white] [label=right:{(b)}] at (8,1.25) {}; \end{tikzpicture}
\caption{The two situations described in the proof of Theorem \ref{thm-connected-patterns}. In part (a), $J^3$ has too many vertices in $S$. In part (b), $J^1$ has too many vertices in $S$.\label{fig-connected-patterns}}\end{center}\end{figure}
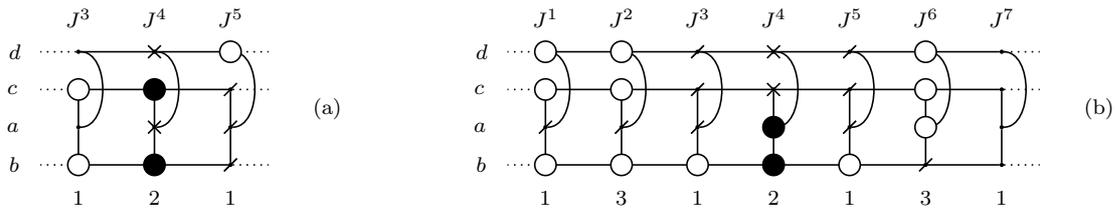

\begin{theorem}Consider the graph $J_n$ for even $n \geq 6$, and a connected dominating set $S$ for $J_n$ such that $|S| = 2n-1$. If $\gamma_c(J_{n-1}) = 2n-3$, and $w^S_4 = 0$, then $S$ does not contain the pattern 112.\label{thm-connected-112}\end{theorem}

\begin{proof}Without loss of generality, suppose that one set of copies meeting the pattern 112 is $J^2, J^3, J^4$, and consider also $J^1$ and $J^5$. Since $J^2$ and $J^3$ have weight 1, we know that $a^2 \not\in S$ and $a^3 \not\in S$. Without loss of generality, suppose that $b^2 \in S$. Then, suppose that $b^3 \in S$ as well. This situation is displayed in part (a) of Figure \ref{fig-connected-112}. Since $S$ is dominating, it must also contain each of $c^1$, $d^1$, $c^4$, and $d^4$. Since $J^4$ has weight 2, we have $a^4 \not\in S$ and $b^4 \not\in S$. Since $S$ is connected, it implies that $b^1 \in S$. Since there are no copies of weight 4 in $S$, this implies that $a^1 \not\in S$. However, by Theorem \ref{thm-connected-remove}, we can then obtain a connected dominating set for $J_{n-1}$ of cardinality $2n-4$, contradicting the assumption that $\gamma_c(J_{n-1}) = 2n-3$. Hence, the assumption that $b^3 \in S$ must be false.

We instead have $b^2 \in S$ and $b^3 \not\in S$. Without loss of generality, suppose that $c^3 \in S$. This situation is displayed in part (b) of Figure \ref{fig-connected-112}. Since $S$ is dominating, it must also contain $d^4$. Since $S$ is connected, we have $c^4 \in S$, and since $J^4$ has weight 2, we have $a^4 \not\in S$ and $b^4 \not\in S$. Then, since $S$ is dominating, we have $b^5 \in S$, and since $S$ is connected we have $c^5 \in S$ and $d^5 \in S$. Since there are no copies of weight 4 in $S$, it implies that $a^5 \not\in S$. However, by Theorem \ref{thm-connected-remove}, we can the obtain a connected dominating set for $J_{n-1}$ of cardinality $2n-4$, which again contradicts the assumption that $\gamma_c(J_{n-1}) = 2n-3$, completing the proof.\end{proof}

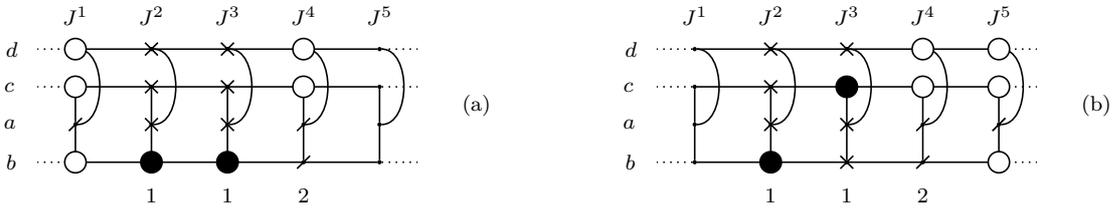
\begin{figure}[h!]\begin{center}\begin{tikzpicture}[smallgraph]\flower{5} \genlab{5} \patt{2}{1} \patt{3}{1} \patt{4}{2} \inS{b}{2} \inS{b}{3} \notS{a}{2} \notS{c}{2} \notS{d}{2} \notS{a}{3} \notS{c}{3} \notS{d}{3} \inSa{b}{1} \inSa{c}{1} \inSa{d}{1} \inSa{c}{4} \inSa{d}{4} \notSa{a}{1} \notSa{a}{4} \notSa{b}{4} \node[white] [label=right:{(a)}] at (6,1.25) {}; \end{tikzpicture} \hfill \begin{tikzpicture}[smallgraph]\flower{5} \genlab{5} \patt{2}{1} \patt{3}{1} \patt{4}{2} \inS{b}{2} \inS{c}{3} \notS{a}{2} \notS{c}{2} \notS{d}{2} \notS{a}{3} \notS{b}{3} \notS{d}{3} \inSa{c}{4} \inSa{d}{4} \inSa{b}{5} \inSa{c}{5} \inSa{d}{5} \notSa{a}{4} \notSa{b}{4} \notSa{a}{5} \node[white] [label=right:{(b)}] at (6,1.25) {}; \end{tikzpicture}
\caption{The two situations described in the proof of Theorem \ref{thm-connected-112}. In both parts, there is a copy with weight 3, but with the $a$ vertex not contained in $S$.\label{fig-connected-112}}\end{center}\end{figure}

In the next theorem, we will require the following concept. Suppose that we have a connected dominating set $S$ for $J_n$. Denote by $J_n(S)$ the subgraph of $J_n$ induced by $S$. By definition, $J_n(S)$ is connected. Then, consider a set of consecutive copies $J^i, \hdots, J^{i+k}$, and define $\overline{S} = S \cap (V(J^i) \cup \hdots \cup V(J^{i+k}))$. If $J_n(\overline{S})$ is a disconnected graph, then we say that the copies $J^i, \hdots, J^{i+k}$ are {\em locally disconnected} in $S$. Since $S$ itself is connected, it is clear that if $S$ contain two sets of locally disconnected consecutive copies, they must overlap by two or more copies.

%\begin{theorem}Consider the graph $J_n$ for even $n \geq 6$, and a connected dominating set $S$ for $J_n$ such that $|S| = 2n-1$. If $\gamma_c(J_{n-2}) = 2n-4$ and $\gamma_c(J_{n-1}) = 2n-3$, then any set of consecutive copies meeting the patterns 113 or 3123 are locally disconnected.\label{thm-connected-local}\end{theorem}

\begin{theorem}Consider the graph $J_n$ for even $n \geq 6$, and a connected dominating set $S$ for $J_n$ such that $|S| = 2n-1$. If $\gamma_c(J_{n-1}) = 2n-3$, then any set of consecutive copies meeting the patterns 113 or 3123 are locally disconnected.\label{thm-connected-local}\end{theorem}

\begin{proof}Consider first the pattern 113. Without loss of generality, suppose that one set of copies meeting this pattern is $J^1, J^2, J^3$, and that this set of copies is not locally disconnected. Suppose that $a^3 \not\in S$. Then, from Theorem \ref{thm-connected-remove} there exists a connected dominating set for $J_{n-1}$ with cardinality $2n-4$, which is impossible since by assumption $\gamma_c(J_{n-1}) = 2n-3$. Hence, $a^3 \in S$. Then, without loss of generality, suppose that $b^3 \in S$, $c^3 \in S$, and $d^3 \not\in S$. This situation is displayed in part (a) of Figure \ref{fig-connected-local}. Since $J^2$ has weight 1, it is clear that $a^2 \not\in S$. Then, in order for $S$ to be dominating, it must contain at least one of $d^1$ and $d^2$. However, if $d^2 \in S$, then $S$ does not contain $b^2$ or $c^2$, and hence $J^1, J^2, J^3$ are locally disconnected, contradicting the initial assumption. Hence, $d^2 \not\in S$, and so $d^1 \in S$. But then, because $J^1$ has weight 1, $S$ does not contain $a^1$, $b^1$, or $c^1$, which again implies that $J^1, J^2, J^3$ are locally disconnected, contradicting the initial assumption. Hence, the initial assumption must be false, and $J^1, J^2, J^3$ are locally disconnected.

Next, consider the pattern 3123. Without loss of generality, suppose that one set of copies meeting this pattern is $J^1, J^2, J^3, J^4$, and that this set of copies if not locally disconnected. Using an identical argument as from the previous paragraph, we must have $a^1 \in S$ and $a^4 \in S$. Then, without loss of generality, suppose that $b^1 \in S$ and $c^1 \in S$. This situation is displayed in part (b) of Figure \ref{fig-connected-local}. Since $J^2$ has weight 1, and the set of copies is not locally disconnected, we have $a^2 \not\in S$ and $d^2 \not\in S$. The remaining two choices are equivalent; without loss of generality, suppose that $b^2 \in S$. Then, since $S$ is dominating, and the set of copies is not locally disconnected, we must have $b^3 \in S$ and $d^3 \in S$. Finally, since $S$ is dominating, and the set of copies is not locally disconnected, we must have $b^4 \in S$, $c^4 \in S$ and $d^4 \in S$. However, this is impossible since $J^4$ has weight 3. Hence, the initial assumption must be false, and $J^1, J^2, J^3, J^4$ are locally disconnected.\end{proof} %Again, it is clear that the removal from $S$ of any vertices from $J^1$ does not avoid this contradiction, and hence an analagous argument can be made for the pattern 2123.\end{proof}

\begin{figure}[h!]\begin{center}\begin{tikzpicture}[smallgraph]\flower{3} \genlab{3} \patt{1}{1} \patt{2}{1} \patt{3}{3} \inS{a}{3} \inS{b}{3} \inS{c}{3} \notS{d}{3} \notS{a}{2} \inSa{d}{1} \notSa{a}{1} \notSa{b}{1} \notSa{c}{1} \notSa{d}{2} \node[white] [label=right:{(a)}] at (4,1.25) {}; \end{tikzpicture} \hspace*{2cm} \begin{tikzpicture}[smallgraph]\flower{4} \genlab{4} \patt{1}{3} \patt{2}{1} \patt{3}{2} \patt{4}{3} \inS{a}{1} \inS{b}{1} \inS{c}{1} \inS{a}{4} \notS{d}{1} \inSa{b}{2} \inSa{b}{3} \inSa{d}{3} \inSa{b}{4} \inSa{c}{4} \inSa{d}{4} \notSa{a}{2} \notSa{c}{2} \notSa{d}{2} \notSa{a}{3} \notSa{c}{3} \node[white] [label=right:{(b)}] at (5,1.25) {}; \end{tikzpicture}
\caption{The two situations described in the proof of Theorem \ref{thm-connected-local}. In part (a), the copies are locally disconnected. In part (b), copy $J^4$ has too many vertices in $S$.\label{fig-connected-local}}\end{center}\end{figure}
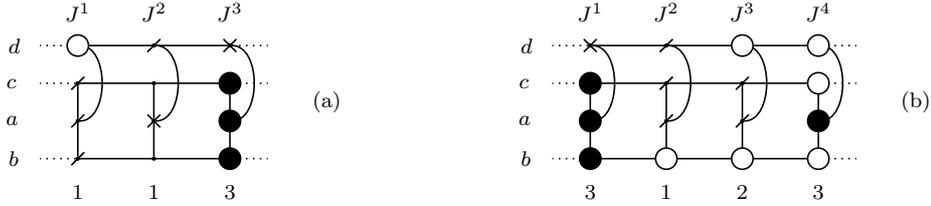

Theorem \ref{thm-connected-patterns}, along with the fact that the patterns 113 and 3123 overlap by at most one copy, leads to the following Corollary.

\begin{corollary}Consider the graph $J_n$ for even $n \geq 6$, and a connected dominating set $S$ for $J_n$ such that $|S| = 2n-1$. If $\gamma_c(J_{n-1}) = 2n-3$, then $S$ contains at most one instance the patterns 113 or 3123.\label{cor-connected-local}\end{corollary}

We are finally ready to prove the main result of this section.

\begin{theorem}Consider the graph $J_n$ for $n \geq 3$. Then,

$$\gamma_c(J_n) = \left\{\begin{aligned}2n\;\;\;, & \mbox{ if } n\mbox{ is even,}\\2n-1, & \mbox{ if } n\mbox{ is odd.}\end{aligned}\right.\label{thm-connected}$$\end{theorem}

\begin{proof}The upper bounds are provided in Section \ref{sec-upper}. We use the MTZ formulation for connected domination from \cite{fan} to confirm that $\gamma_c(J_3) = 5, \gamma_c(J_4) = 8, \gamma_c(J_5) = 9$, and $\gamma_c(J_6) = 12$. Suppose there is some value $k \geq 7$ such that Theorem \ref{thm-connected} is true for $n = 3, \hdots, k-1$, but false for $n = k$. We will first consider the case when $k$ is odd. Then there is a connected dominating set $S$ for $J_k$ such that $|S| = 2k-2$. Clearly, $S$ contains a copy, say $J^i$, with weight 1, and $a^i \not\in S$. Then from Theorem \ref{thm-connected-remove}, it implies that there is a connected dominating set for $J_{k-1}$ with cardinality $2k-3$, which contradicts the assumption that Theorem \ref{thm-connected} is true for $J_{k-1}$.

Hence, $k$ must be even. Then, there is a connected dominating set $S$ for $J_k$ such that $|S| = 2k-1$. Suppose that there are two copies of $J_k$ with weight 2 in $S$. Then, from Theorem \ref{thm-connected-remove} we can obtain a connected dominating set for $J_{k-2}$ with cardinality $2k-5$, which contradicts the initial assumption. Hence, there must be at most one copy of $J_k$ with weight 2 in $S$. That is, either $w^S_2 = 0$ or $w^S_2 = 1$.

Suppose that $w^S_2 = 0$. Then we have $w^S_1 + w^S_3 + w^S_4 = k$, and $w^S_1 + 3w^S_3 + 4w^S_4 = 2k-1$. From the latter equation, it is clear that $w^S_1$ and $w^S_3$ must have different parity. Hence, from the former equation, $w^S_4$ must be odd. Combining the two equations, we obtain $2w^S_1 = k+1+w^S_4$. Note that this implies that more than half of the copies have weight 1 in $S$. From Theorem \ref{thm-connected-patterns}, $S$ cannot contain the pattern 111. Hence, it can be seen that $S$ contains the pattern 11 at least $w^S_4+1$ times. From Corollary \ref{cor-connected-local} we know there is at most one instance of the pattern 113 in $S$. Hence, there is at least $w^S_4$ instances of the pattern 11 where both adjacent copies have weight 4, which is impossible.

Hence, we must have $w^S_2 = 1$. Then, we have $w^S_1 + w^S_3 + w^S_4 = k-1$, and $w^S_1 + 3w^S_3 + 4w^S_4 = 2k-3$. From the latter equation, it is clear that $w^S_1$ and $w^S_3$ must have different parity. Hence, from the former equation, $w^S_4$ must be even. Combining the two equations, we obtain $2w^S_1 = k + w^S_4$. Similar to before, this implies that at least half of the copies have weight 1 in $S$. From Theorem \ref{thm-connected-patterns}, $S$ cannot contain the pattern 111. Hence, it can be seen that $S$ contains the pattern 11 at least $w^S_4$ times. Now, suppose that $S$ contains the pattern 4114. From Lemma \ref{lem-connected-4}, we can then obtain a connected dominating set for $J_{k-3}$ with cardinality $2k-7$ which has one fewer copy of weight 4 than $S$, which in turn violates Lemma \ref{lem-connected-odd}. Hence, every instance of the pattern 11 has a copy of weight 2 or 3 next to it, and from Corollary \ref{cor-connected-local} this means there is at most one instance of the pattern 11. This, in turn, implies that $w^S_4 \leq 1$, and since $w^S_4$ is even, we have $w^S_4 = 0$. Hence, $w^S_1 = \frac{k}{2}$. If there are no instances of the pattern 11, then $S$ must contain the pattern 1312131, violating Theorem \ref{thm-connected-patterns}. Hence there is exactly one instance of the pattern 11. This means that there must be one other instance in $S$ of two consecutive copies having non-unit weight. The only options are that $S$ contains the pattern 23, or the pattern 33.

Suppose $S$ contains the pattern 23. Without loss of generality, suppose that the copies $J^3, J^4$ meet this pattern. Since this is the only instance of $S$ having two consecutive copies of weight other than 1, we know that $J^2$ has weight 1. From Theorem \ref{thm-connected-112} we know that $J^1$ cannot have weight 1, or else $S$ would contain the pattern 112. The only remaining option is that $J^1$ has weight 3, and so copies $J^1, J^2, J^3, J^4$ meet the pattern 3123. Then, since $J^3$ has weight 2 and $w^S_2 = 1$, there are no other copies with weight 2. Hence, the one instance of the pattern 11 which occurs elsewhere in the graph must be followed by a copy of weight 3. That is, $S$ contains both the patterns 113 and 3123, which from Corollary \ref{cor-connected-local} is impossible. Therefore, $S$ must not contain the pattern 23, and instead contains the pattern 33.

Finally, without loss of generality, suppose that the copies $J^1, J^2$ meet the pattern 33. There is one vertex from $J^1$ which is not contained in $S$. Suppose we define $S_2$ which is equal to the union of $S$ and this one vertex. Hence, $S_2$ is a connected dominating set for $J_k$ with cardinality $2k$. Then, from Lemma \ref{lem-connected-4}, we can obtain a connected dominating set for $J_{k-1}$ with cardinality $2k-3$, which contains exactly one copy with weight 4, contradicting Lemma \ref{lem-connected-odd}.

Hence, in all cases a contradiction is reached, completing the proof.\end{proof}

\section{Weak Roman Domination and Secure Domination}

In this section, we determine (simultaneously) the weak Roman domination numbers and the secure domination numbers for flower snarks. Recall that for the weak Roman domination number, rather than seeking to construct a set $S$, we instead look to define a function $f : V \rightarrow \{0, 1, 2\}$. In order to use language consistent with the rest of this paper, in this section we define the weight of a copy $J^i$ to be equal to $f(a^i) + f(b^i) + f(c^i) + f(d^i)$. Note that this is somewhat more ambiguous than in previous sections; for instance, if a copy has weight 2, it may have two vertices with weight 1, or a single vertex with weight 2. We will deal with these ambiguities when they arise. Similarly to in previous sections, we will say that $f$ contains a pattern if there is a set of consecutive copies whose weights meet that pattern.

Although there is a technical definition for weak Roman domination (e.g. see Definition \ref{def-wrd}), it is useful to provide an intuitive interpretation. Recall that for standard domination, one interpretation is that the set of vertices needs to be protected. If a guard is placed at a vertex, then it protects that vertex, and also all adjacent vertices. A configuration of guards which protects all vertices corresponds to a dominating set. A similar interpretation applies for weak Roman domination. Suppose that at each vertex we can place up to two guards. As before, a vertex is protected if there is at least one guard there, or there is at least one guard among its adjacent vertices. Then we further consider the notion of a vertex being {\em defended}. We say a vertex $v$ is defended if there is at least one guard at $v$, or alternatively, if it possible to relocate one guard from an adjacent vertex $w$ to $v$, in such a way that all vertices are still protected by the resulting configuration of guards. In the latter case, we will say that $w$ {\em can defend} $v$. Then a weak Roman dominating function is one in which every vertex is defended.

Throughout this section, we will often use an argument which goes as follows; suppose there is exactly one guard at $v$, and that $v$ defends another adjacent vertex $w$. In order to do so, the guard from $v$ would move to $w$. However, doing so may mean another vertex $u$, also adjacent to $v$, becomes unprotected. In such a case, we will say that {\em $v$ cannot defend $w$ without leaving $u$ unprotected.} Another alternative is as follows; suppose again that $v$ defends $w$. In order to do so, the guard from $v$ would move to $w$. However, doing so may mean that $u$ will be unprotected unless there is a guard at another vertex, say $x$. In such a case, we will say {\em in order for $v$ to defend $w$, there must be a guard at $x$ to avoid leaving $u$ unprotected}.

Although the following result is simple, we include it here as it will be used many times in this section.

\begin{lemma}Consider the graph $J_n$, for $n \geq 3$, and a weak Roman dominating function $f$. If a copy $J^i$ has weight 1 in $f$, then none of the vertices from $J^i$ can defend any vertices from $J^{i-1}$ or $J^{i+1}$.\label{lem-wrd-1}\end{lemma}

\begin{proof}Since $J^i$ has weight 1, there is exactly one vertex with positive weight. If $f(a^i) = 1$, then the result follows immediately since $a^i$ is not adjacent to any vertices from $J^{i-1}$ or $J^{i+1}$. Then, suppose instead that $f(a^i) = 0$ and one of the other vertices in $J^i$ has weight 1. Then that vertex cannot defend any vertex from $J^{i-1}$ or $J^{i+1}$ without leaving $a^i$ unprotected.\end{proof}

\begin{theorem}Consider the graph $J_n$, for $n \geq 3$, and a weak Roman dominating function $f$. Then $f$ does not contain the patterns 111, 1121 or 1212121.\label{thm-wrd-patterns}\end{theorem}

\begin{proof}Suppose that $f$ contains the pattern 111. Without loss of generality, suppose that copies $J^1, J^2, J^3$ meet the pattern 111 in $f$. This situation is displayed in part (a) of Figure \ref{fig-wrd-patterns}.  Each vertex in $J^2$ must be defended, but since $J^1$ and $J^3$ have weight 1, from Lemma \ref{lem-wrd-1} none of their vertices can defend any vertices from $J^2$. Hence, each vertex in $J^2$ must be defended solely by vertices in $J^2$. It is clear that this implies that $f(a^2) = 1$. Then, consider $b^2$. In order for $a^2$ to defend $b^2$, there must be a guard at either $c^1$ or $c^3$ to avoid leaving $c^2$ unprotected. Likewise, in order for $a^2$ to defend $b^2$, there must be a guard at either $d^1$ or $d^3$ to avoid leaving $d^2$ unprotected. Since $J^1$ and $J^3$ have weight 1, at most one vertex from each can have a guard. Without loss of generality, suppose that $f(c^1) = 1$ and $f(d^3) = 1$. Then $a^2$ cannot defend $c^2$ without leaving $b^2$ unprotected. Hence, this is impossible, and $f$ does not contain the pattern 111.

Next, suppose that $f$ contains the pattern 1121. Without loss of generality, suppose that copies $J^1$, $J^2$, $J^3$, $J^4$ meet the pattern 1121 in $f$. Since $J^2$ and $J^4$ have weight 1, from Lemma \ref{lem-wrd-1} none of their vertices can defend any vertices from $J^3$. Hence, each vertex from $J^3$ must be defended solely by vertices in $J^3$. Clearly, this implies that $f(a^3) \geq 1$. There are two possibilities, either $f(a^3) = 2$, or $f(a^3) = 1$ and there is a guard at another vertex in $J^3$. Suppose first that $f(a^3) = 2$. Then there are no vertices from $J^1$ or $J^3$ that can defend any vertices from $J^2$. Hence, $J^2$ is in an equivalent situation to $J^2$ in the previous paragraph and so this situation is impossible. Instead, suppose that $f(a^3) = 1$ and, without loss of generality, that $f(b^3) = 1$. This situation is displayed in part (b) of Figure \ref{fig-wrd-patterns}. Clearly, vertices $c^2$ and $d^2$ can only be defended by vertices from $J^2$, and hence we must have $f(a^2) = 1$. Then, in order for $a^2$ to defend $c^2$, there must be a guard at $d^1$ to avoid leaving $d^2$ unprotected. Similarly, in order for $a^2$ to defend $d^2$, there must be a guard at $c^1$ to avoid leaving $c^2$ unprotected. This is impossible since $J^1$ has weight 1, and so $f$ does not contain the pattern 1121.

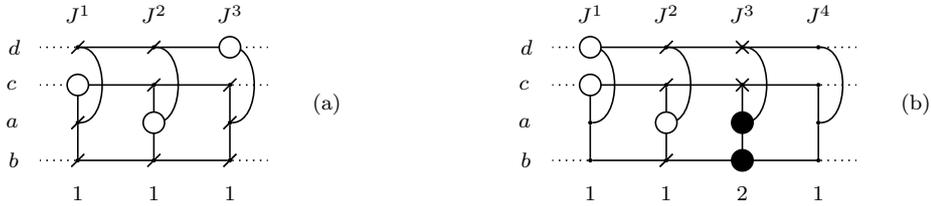
\begin{figure}[h!]\begin{center}\begin{tikzpicture}[smallgraph]\flower{3} \genlab{3} \patt{1}{1} \patt{2}{1} \patt{3}{1} \inSa{c}{1} \inSa{a}{2} \inSa{d}{3} \notSa{a}{1} \notSa{b}{1} \notSa{d}{1} \notSa{b}{2} \notSa{c}{2} \notSa{d}{2} \notSa{a}{3} \notSa{b}{3} \notSa{c}{3} \node[white] [label=right:{(a)}] at (4,1.25) {}; \end{tikzpicture} \hspace*{2cm} \begin{tikzpicture}[smallgraph]\flower{4} \genlab{4} \patt{1}{1} \patt{2}{1} \patt{3}{2} \patt{4}{1} \inS{a}{3} \inS{b}{3} \notS{c}{3} \notS{d}{3} \inSa{c}{1} \inSa{d}{1} \inSa{a}{2} \notSa{b}{2} \notSa{c}{2} \notSa{d}{2} \node[white] [label=right:{(b)}] at (5,1.25) {}; \end{tikzpicture}
\caption{The first two situations described in the proof of Theorem \ref{thm-wrd-patterns}. In part (a), $a^2$ cannot defend $c^2$ without leaving $b^2$ unprotected. In part (b), there are too many guards in $J^1$.\label{fig-wrd-patterns}}\end{center}\end{figure}

Finally, suppose that $f$ contains the pattern 1212121. Without loss of generality, suppose that copies $J^1, \hdots, J^7$ meet the pattern 1212121 in $f$. Since $J^1$ and $J^3$ have weight 1, from Lemma \ref{lem-wrd-1} none of their vertices can defend any vertices from $J^2$. Hence, each vertex from $J^2$ must be defended solely by vertices in $J^2$, which implies that $f(a^2) \geq 1$. Analogous arguments imply that $f(a^4) \geq 1$ and $f(a^6) \geq 1$. Now, consider $J^2$. There are two possibilities, either $f(a^2) = 2$, or $f(a^2) = 1$ and there is a guard at another vertex in $J^2$. Suppose first that $f(a^2) = 2$. Then there are no vertices from $J^2$ that can defend any vertices from $J^3$, and at most one vertex from $J^4$ can defend a vertex from $J^3$. Hence, $J^3$ is in an equivalent situation to $J^2$ in the previous paragraph, and so this situation is impossible.

Instead, suppose that $f(a^2) = 1$ and, without loss of generality, that $f(b^2) = 1$. This situation is displayed in Figure \ref{fig-wrd-patterns2}. Then, in order for $a^2$ to defend $c^2$, there must be a guard at either $d^1$ or $d^3$ to avoid leaving $d^2$ unprotected. Similarly, in order for $a^2$ to defend $d^2$, there must be a guard at either $c^1$ or $c^3$ to avoid leaving $c^2$ unprotected. Since $J^1$ and $J^3$ have weight 1, at most one vertex from each may have a guard. Without loss of generality, suppose that $f(c^1) = 1$ and $f(d^3) = 1$. Then, since $f$ is dominating, we must have $f(c^4) = 1$ in order to protect $c^3$. Now, in order for $a^4$ to defend $d^4$, there must be a guard at $b^5$ to avoid leaving $b^4$ unprotected. Then, since $f$ is dominating, we must have $f(d^6) = 1$ in order to protect $d^5$. Finally, consider vertex $c^5$. There is only one adjacent vertex with a guard that can defend it, $c^4$. However, $c^4$ cannot defend $c^5$ without leaving $c^3$ unprotected. Thus, $c^5$ is not defended. This is impossible, and hence $f$ does not contain the pattern 1212121.\end{proof}

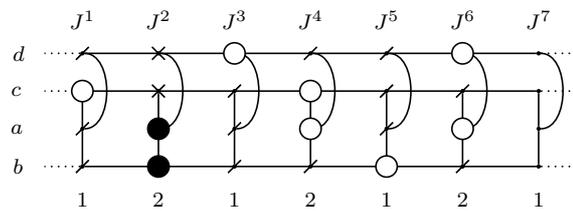
\begin{figure}[h!]\begin{center}\begin{tikzpicture}[smallgraph]\flower{7} \genlab{7} \patt{1}{1} \patt{2}{2} \patt{3}{1} \patt{4}{2} \patt{5}{1} \patt{6}{2} \patt{7}{1} \inS{a}{2} \inS{b}{2} \notS{c}{2} \notS{d}{2} \inSa{c}{1} \inSa{d}{3} \inSa{a}{4} \inSa{c}{4} \inSa{b}{5} \inSa{a}{6} \inSa{d}{6} \notSa{a}{1} \notSa{b}{1} \notSa{d}{1} \notSa{a}{3} \notSa{b}{3} \notSa{c}{3} \notSa{b}{4} \notSa{d}{4} \notSa{a}{5} \notSa{c}{5} \notSa{d}{5} \notSa{b}{6} \notSa{c}{6}\end{tikzpicture}
\caption{The third situation described in the proof of Theorem \ref{thm-wrd-patterns}. Here, $c^4$ cannot defend $c^5$ without leaving $c^3$ unprotected.\label{fig-wrd-patterns2}}\end{center}\end{figure}

\begin{lemma}Consider the graph $J_n$, for $n \geq 4$, and a weak Roman dominating function $f$ with weight $w(f) \leq \frac{3n}{2}$. Then any set of four consecutive copies have a combined weight of 6 in $f$, every copy has weight either 1 or 2 in $f$, and $w(f) = \frac{3n}{2}$.\label{lem-wrd-weight6}\end{lemma}

\begin{proof}Denote by $w^4(i)$ the combined weights of $J^i, J^{i+1}, J^{i+2}, J^{i+3}$ where the superscripts are taken modulo $n$. Now, by Corollary \ref{cor-weight1} we know that no copies have weight 0, and so $w^4(i) \geq 4$. Suppose $w^4(i) = 4$. This implies that $J^i, J^{i+1}, J^{i+2}, J^{i+3}$ all have weight 1. However, this means $f$ contains the pattern 111, which by Theorem \ref{thm-wrd-patterns} is impossible. Then, suppose $w^4(i) = 5$. Up to symmetry, there are only two possibilities; copies $J^i, J^{i+1}, J^{i+2}, J^{i+3}$ either have the pattern 1112 or 1121. Both of these options are impossible by Theorem \ref{thm-wrd-patterns}. Hence, we have $w^4(i) \geq 6$. Then, by assumption, we have $\sum_{i=1}^n w^4(i) \leq 6n$. This is only possible if all $w^4(i) = 6$, and also implies that $w(f) = \frac{3n}{2}$.

Since each copy has weight at least one, and each set of four consecutive copies has a combined weight of six, it is clear that any individual copy must have weight at most three. Suppose there is a copy $J^i$ with weight 3 in $f$, then the next three copies must each have weight 1. However, by Theorem \ref{thm-wrd-patterns} this is impossible. Hence, each copy $J^i$ must have weight either 1 or 2 in $f$.\end{proof}

\begin{lemma}Consider the graph $J_n$, for $n \geq 3$, and a weak Roman dominating function $f$ with weight $w(f)$. Suppose there is an integer $k \geq 3$ such that in $f$, the configuration of guards in copies $J^{i+1}, \hdots, J^{i+k}$ is identical to the configuration of guards in $J^{i+k+1}, \hdots, J^{i+2k}$, and suppose that $J^{i+1}, \hdots, J^{i+k}$ collectively contain $g$ guards. Then there is a weak Roman dominating function for $J_{n-k}$ with weight equal to $w(f) - g$.\label{lem-wrd-overlap}\end{lemma}

\begin{proof}Suppose that we ``smooth out" copies $J^{i+1}, \hdots, J^{i+k}$ in the way displayed in Figure \ref{fig-connected-GG2}. That is, we add edges connecting $b^i$ to $b^{i+k+1}$, $c^i$ to $c^{i+k+1}$, and $d^i$ to $d^{i+k+1}$, and then we delete the copies $J^{i+1}, \hdots, J^{i+k}$. Call the resulting graph $G_2$. It is easy to check that $G_2$ is isomorphic to $J_{n-k}$. Then, define a new function $f_2$ such that $f_2(v) = f(v)$ if $v \in V \setminus \{V(J^{i+1}) \cup \hdots \cup V(J^{i+k})\}$, and $f_2(v)$ is undefined otherwise. Since $J^{i+1}, \hdots, J^{i+k}$ collectively contain $g$ guards in $f$, it is clear that $w(f_2) = w(f) - g$. Then if we can show that $f_2$ is a weak Roman dominating function for $G_2$, the result follows immediately.

For a given function, one can check to see if a vertex $v$ is defended by observing the configuration of guards in vertices no more than distance three away from $v$. Then, for any vertex $v$ in $G_2$, consider the configuration of graphs (according to $f_2)$ in the subgraph induced by the vertices at most distance 3 from $v$. By assumption, this is identical to the corresponding configuration of guards (according to $f$) in the subgraph induced by the vertices at most distance 3 from $v$ in $J_n$. Hence, all vertices of $G$ are defended in $f_2$, and so $f_2$ is a weak Roman dominating function for $G$, leading to the desired result.\end{proof}

In the proof of the following theorem, whenever there are two guards at a vertex, we will mark that vertex with a \begin{tikzpicture}[smallgraph,scale=0.8]\inSS{b}{1}\end{tikzpicture}.

\begin{theorem}Consider the graph $J_n$, for $n \geq 9$, and a weak Roman dominating function $f$. If $f$ contains the pattern 21122112, then there exists a weak Roman dominating function for $J_{n-4}$ with weight equal to $w(f) - 6$.\label{thm-wrd-1122}\end{theorem}

\begin{proof}Without loss of generality, suppose that copies $J^1, \hdots, J^{8}$ meet the pattern 21122112 in $f$. Consider copy $J^4$, and suppose that $f(a^4) = 2$. This situation is displayed in part (a) of Figure \ref{fig-wrd-1122}. Then, no vertices from $J^4$ can defend any vertices from $J^3$, and by Lemma \ref{lem-wrd-1} the same is also true for $J^2$. Hence, all four vertices in $J^3$ must be defended by vertices in $J^3$. This implies that $f(a^3) = 1$. Then, in order for $a^3$ to defend $b^3$, there must be a guard at $c^2$ to avoid leaving $c^3$ unprotected, and there must also be a guard at $d^2$ to avoid leaving $d^3$ unprotected. Since $J^2$ has weight 1, this is impossible, and so $f(a^4) \neq 2$.

Then, suppose that $f(a^4) = 1$. Then there is another vertex in $J^4$ with a guard. Without loss of generality, suppose that $f(b^4) = 1$. This situation is displayed in part (b) of Figure \ref{fig-wrd-1122}. Using a similar argument as in the previous paragraph, vertices $c^3$ and $d^3$ must be defended by vertices from $J^3$, and hence $f(a^3) = 1$. Then, in order for $a^3$ to defend $c^3$, there must be a guard at $d^2$ to avoid leaving $d^3$ unprotected. Likewise, in order for $a^3$ to defend $d^3$, there must be a guard at $c^2$ to avoid leaving $c^3$ unprotected. Since $J^2$ has weight 1, this is impossible. Hence we can conclude that $f(a^4) = 0$. It is clear that identical arguments can be made to conclude that $f(a^1) = f(a^5) = f(a^8) = 0$ as well.

Then, suppose that there is a vertex in $J^4$ with two guards. Without loss of generality, suppose that $f(b^4) = 2$. This situation is displayed in part (c) of Figure \ref{fig-wrd-1122}. An identical argument to that in the previous paragraph can be used to show that $f(a^3) = 1$, and this again implies that $f(c^2) = f(d^2) = 1$. Since $J^2$ has weight 1, this is impossible. Hence, exactly two vertices in $J^4$ have weight 1. Again, identical arguments can be made to reach the same conclusion for each of $J^1$, $J^5$ and $J^8$.

Now, without loss of generality, suppose that $f(b^4) = f(c^4) = 1$. Then, $d^4$ must be defended by either $d^3$ or $d^5$, and by Lemma \ref{lem-wrd-1} it cannot be $d^3$. Hence, we must have $f(d^5) = 1$, and there must be one more vertex in $J^5$ with a guard, either $b^5$ or $c^5$. Due to symmetry, either choice may be made without loss of generality; we will choose $f(b^5) = 1$. Now, consider vertex $d^3$. From Lemma \ref{lem-wrd-1} it cannot be defended by $d^2$. Hence, we must have either $f(a^3) = 1$ or $f(d^3) = 1$.

Suppose that $f(a^3) = 1$. Then, $c^3$ must be defended by at least one of $a^3$ or $c^4$ (from Lemma \ref{lem-wrd-1} it cannot be defended by $c^2$). Suppose first that $a^3$ defends $c^3$. This situation is displayed in part (d) of Figure \ref{fig-wrd-1122}. In order for $a^3$ to defend $c^3$, there must be a guard at $d^2$ to avoid leaving $d^3$ unprotected. Then, in order to defend $b^2$ and $c^2$ we must have $f(b^1) = 1$ and $f(c^1) = 1$ respectively. But then, $b^1$ cannot defend $a^1$ without leaving $b^2$ unprotected, and $c^1$ cannot defend $a^1$ without leaving $c^2$ unprotected. Hence, $a^1$ is not defended, which is impossible, and so we conclude that $a^3$ cannot defend $c^3$. Hence, $c^4$ must defend $c^3$. This situation is displayed in part (e) of Figure \ref{fig-wrd-1122}. In order for $c^4$ to defend $c^3$, there must be a guard at $c^6$ to avoid leaving $c^5$ unprotected. Also, the only vertex which can defend $d^4$ is $d^5$, but in order to do so, there must be a guard at $d^7$ to avoid leaving $d^6$ unprotected. Finally, consider $a^5$, which can only be defended by either $b^5$ or $d^5$. However, $b^5$ cannot defend $a^5$ without leaving $b^6$ unprotected, and $d^5$ cannot defend $a^5$ without leaving $d^4$ unprotected. Hence, all of these cases are impossible, and so conclude that $f(a^3) = 0$, and accordingly, $f(d^3) = 1$.

The current situation is displayed in part (f) of Figure \ref{fig-wrd-1122}. Now consider $c^3$. From Lemma \ref{lem-wrd-1} it is clear that $c^2$ cannot defend $c^3$, so $c^4$ must defend $c^3$. In order for $c^4$ to defend $c^3$, there must be a guard at $c^6$ to avoid leaving $c^5$ unprotected. Likewise, from Lemma \ref{lem-wrd-1}, $d^3$ cannot defend $d^4$, so $d^5$ must defend $d^4$. In order for $d^5$ to defend $d^4$, there must be a guard at $d^7$ to avoid leaving $d^6$ unprotected. Furthermore, from Lemma \ref{lem-wrd-1}, $c^6$ cannot defend $c^5$, so $c^4$ must defend $c^5$. In order for $c^4$ to defend $c^5$, there must be a guard at $c^2$ to avoid leaving $c^3$ unprotected. Finally, from Lemma \ref{lem-wrd-1}, the only vertices which can defend $b^2$ and $d^2$ are $b^1$ and $d^1$ respectively, and so $f(b^1) = f(d^1) = 1$. Likewise, from Lemma \ref{lem-wrd-1}, the only vertices which can defend $b^7$ and $c^7$ are $b^8$ and $c^8$ respectively, so $f(b^8) = f(c^8) = 1$.

At this point, it can be seen that the configuration of guards in $J^1, J^2, J^3, J^4$ is identical to the configuration of guards in $J^5, J^6, J^7, J^8$. Then, the result follows immediately from Lemma \ref{lem-wrd-overlap}.\end{proof}

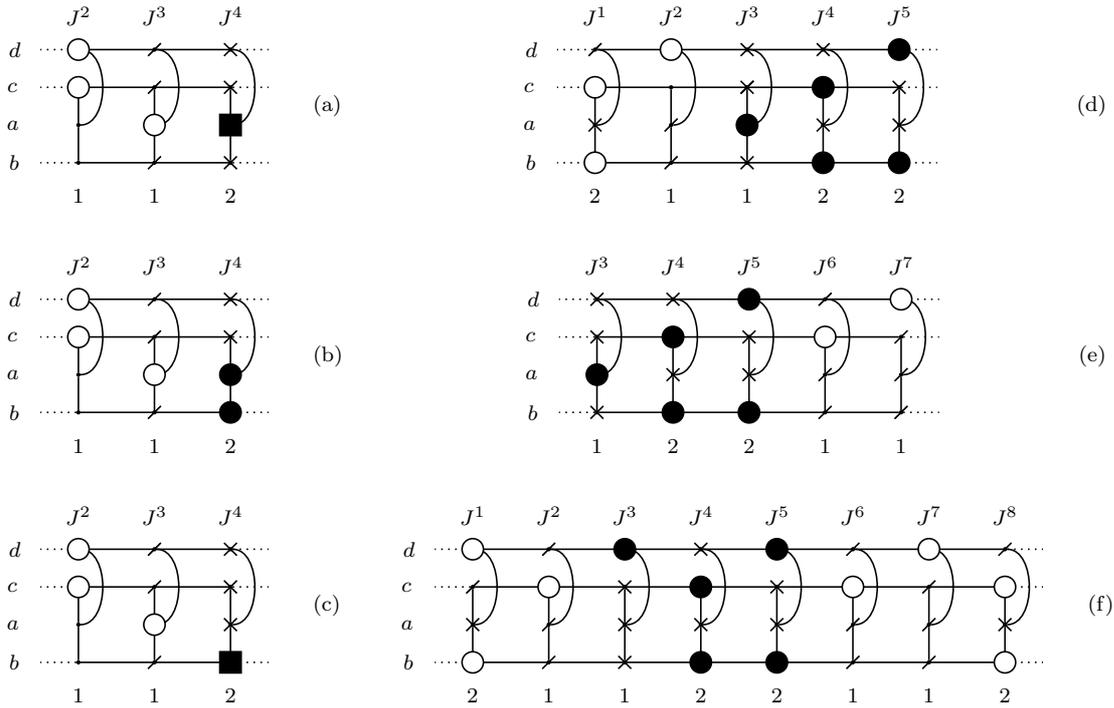
\begin{figure}[h!]\begin{center}
\begin{tikzpicture}[smallgraph]\flower{3} \lab{1}{$J^2$} \lab{2}{$J^3$} \lab{3}{$J^4$} \patt{1}{1} \patt{2}{1} \patt{3}{2} \inSS{a}{3} \notS{b}{3} \notS{c}{3} \notS{d}{3} \inSa{c}{1} \inSa{d}{1} \inSa{a}{2} \notSa{b}{2} \notSa{c}{2} \notSa{d}{2} \node[white] [label=right:{(a)}] at (4,1.25) {}; \end{tikzpicture}
\hfill \begin{tikzpicture}[smallgraph]\flower{5} \genlab{5} \patt{1}{2} \patt{2}{1} \patt{3}{1} \patt{4}{2} \patt{5}{2} \inS{a}{3} \inS{b}{4} \inS{c}{4} \inS{b}{5} \inS{d}{5} \notS{a}{1} \notS{b}{3} \notS{c}{3} \notS{d}{3} \notS{a}{4} \notS{d}{4} \notS{a}{5} \notS{c}{5} \inSa{b}{1} \inSa{c}{1} \inSa{d}{2} \notSa{d}{1} \notSa{a}{2} \notSa{b}{2} \notSa{c}{3} \node[white] [label=right:{(d)}] at (7.25,1.25) {}; \end{tikzpicture} \hspace*{3.25cm}\\
\vspace*{0.5cm}\begin{tikzpicture}[smallgraph]\flower{3} \lab{1}{$J^2$} \lab{2}{$J^3$} \lab{3}{$J^4$} \patt{1}{1} \patt{2}{1} \patt{3}{2} \inS{a}{3} \inS{b}{3} \notS{c}{3} \notS{d}{3} \inSa{c}{1} \inSa{d}{1} \inSa{a}{2} \notSa{b}{2} \notSa{c}{2} \notSa{d}{2} \node[white] [label=right:{(b)}] at (4,1.25) {}; \end{tikzpicture}
\hfill \begin{tikzpicture}[smallgraph]\flower{5} \lab{1}{$J^3$} \lab{2}{$J^4$} \lab{3}{$J^5$} \lab{4}{$J^6$} \lab{5}{$J^7$} \patt{1}{1} \patt{2}{2} \patt{3}{2} \patt{4}{1} \patt{5}{1} \inS{a}{1} \inS{b}{2} \inS{c}{2} \inS{b}{3} \inS{d}{3} \notS{b}{1} \notS{c}{1} \notS{d}{1} \notS{a}{2} \notS{d}{2} \notS{a}{3} \notS{c}{3} \inSa{c}{4} \inSa{d}{5} \notSa{a}{4} \notSa{b}{4} \notSa{d}{4} \notSa{a}{5} \notSa{b}{5} \notSa{c}{5} \node[white] [label=right:{(e)}] at (7.25,1.25) {}; \end{tikzpicture} \hspace*{3.25cm}\\
\vspace*{0.5cm}\begin{tikzpicture}[smallgraph]\flower{3} \lab{1}{$J^2$} \lab{2}{$J^3$} \lab{3}{$J^4$} \patt{1}{1} \patt{2}{1} \patt{3}{2} \inSS{b}{3} \notS{a}{3} \notS{c}{3} \notS{d}{3} \inSa{c}{1} \inSa{d}{1} \inSa{a}{2} \notSa{b}{2} \notSa{c}{2} \notSa{d}{2} \node[white] [label=right:{(c)}] at (4,1.25) {}; \end{tikzpicture} \hfill \begin{tikzpicture}[smallgraph]\flower{8} \genlab{8} \patt{1}{2} \patt{2}{1} \patt{3}{1} \patt{4}{2} \patt{5}{2} \patt{6}{1} \patt{7}{1} \patt{8}{2} \inS{d}{3} \inS{b}{4} \inS{c}{4} \inS{b}{5} \inS{d}{5} \notS{a}{1} \notS{a}{3} \notS{b}{3} \notS{c}{3} \notS{a}{4} \notS{d}{4} \notS{a}{5} \notS{c}{5} \notS{a}{8} \inSa{b}{1} \inSa{d}{1} \inSa{c}{2} \inSa{c}{6} \inSa{d}{7} \inSa{b}{8} \inSa{c}{8} \notSa{c}{1} \notSa{a}{2} \notSa{b}{2} \notSa{d}{2} \notSa{a}{6} \notSa{b}{6} \notSa{d}{6} \notSa{a}{7} \notSa{b}{7} \notSa{c}{7} \notSa{d}{8} \node[white] [label=right:{(f)}] at (9,1.25) {}; \end{tikzpicture}
\caption{The six situations described in the proof of Theorem \ref{thm-wrd-1122}. In parts (a), (b) and (c), there are too many guards in copy $J^2$. In part (d), $a^1$ is undefended. In part (e), $a^5$ is undefended. In part (f), the configuration of guards in $J^1, J^2, J^3, J^4$ is identical to that in $J^5, J^6, J^7, J^8$.\label{fig-wrd-1122}}\end{center}\end{figure}

We are finally ready to prove the main result of this section.

\begin{theorem}Consider the graph $J_n$, for $n \geq 3$. Then, $\gamma_s(J_n) = \gamma_r(J_n) = \left\lceil\displaystyle\frac{3n+1}{2}\right\rceil.$\label{thm-wrd}\end{theorem}

\begin{proof}The (coincident) upper bounds for both weak Roman domination and secure domination are provided in Section \ref{sec-upper}. Note that $\gamma_r(J_n) \leq \gamma_s(J_n)$, so a corresponding lower bound for weak Roman domination will also serve as a lower bound for secure domination. We use the formulation for weak Roman domination from \cite{burger} to confirm that $\gamma_r(J_3) = 5$, $\gamma_r(J_4) = 7$, $\gamma_r(J_5) = 8$, $\gamma_r(J_6) = 10$, $\gamma_r(J_7) = 11$, and $\gamma_r(J_8) = 13$, and the formulation from \cite{burdett} to confirm the corresponding results for secure domination. Then, suppose there is a value $k \geq 9$ such that Theorem \ref{thm-wrd} is true for $n = 3, \hdots, k-1$, but is not true for $n = k$. That is, there exists a weak Roman dominating function $f$ for $J_k$ with weight $w(f) < \left\lceil \frac{3k+1}{2} \right\rceil$. This implies that $w(f) \leq \frac{3k}{2}$. Hence, from Lemma \ref{lem-wrd-weight6} we have $w(f) = \frac{3k}{2}$, each set of four consecutive copies has weight 6, and each copy has weight 1 or 2.

It is easy to check that there are only two possibilities. Either $f$ has the repeating pattern $121212 \hdots 12$, or $f$ has the repeating pattern $21122112 \hdots 2112$. Note that it is impossible to combine these two patterns without there being a set of four consecutive copies with weight not equal to 6. Then, since $k \geq 9$, by Theorem \ref{thm-wrd-patterns} the pattern $121212 \hdots 12$ is impossible. Hence, $f$ must have the repeating pattern $21122112 \hdots 2112$. Then, by Theorem \ref{thm-wrd-1122} there is weak Roman dominating function for $J_{k-4}$ with weight equal to $\frac{3k}{2} - 6 < \left\lceil\frac{3(k-4)+1}{2}\right\rceil $. This contradicts our initial assumption, completing the proof.\end{proof}

\bibliographystyle{plain}

\end{document}